\documentclass[twoside,a4paper,11pt]{amsart}
\usepackage{graphicx}
\usepackage{color}
\usepackage[rose,frak,hyperref,operator]{paper_diening}
\usepackage[notcite,notref]{showkeys}

 \def\Xint#1{\mathchoice
 {\XXint\displaystyle\textstyle{#1}}%
 {\XXint\textstyle\scriptstyle{#1}}%
 {\XXint\scriptstyle\scriptscriptstyle{#1}}%
 {\XXint\scriptscriptstyle\scriptscriptstyle{#1}}%
 \!\int}
 \def\XXint#1#2#3{{\setbox0=\hbox{$#1{#2#3}{\int}$}
 \vcenter{\hbox{$#2#3$}}\kern-.5\wd0}}
 
 \def\dashint{\Xint-}

\def\N{\mathbb{N}}

\def\diagint{{\raise-.1pt\hbox{--}\hskip-7.9pt\intop}}

\newtheorem{theorem}{Theorem}[section]
\newtheorem{lemma}[theorem]{Lemma}
\newtheorem{corollary}[theorem]{Corollary}

\newtheorem{rem}[theorem]{Remark}

\newtheorem{definition}[theorem]{Definition}

\newtheorem{assumption}[theorem]{Assumption}

\numberwithin{equation}{section}
\newcommand{\meantmp}[2]{#1\langle{#2}#1\rangle}

\newcommand{\hL}{\mathring L}

\newcommand{\Rn}{{\setR^n}}
\newcommand{\RN}{{\setR^n}}
\newcommand{\RNn}{\setR^{n\times n}}

\newcommand{\D}{\varepsilon (v) }
\newcommand{\Dw}{\varepsilon (w) }

\newcommand{\divv}{\divergence}

\newcommand{\Bog}{\ensuremath{\text{\rm Bog}}}

\newcommand{\dx}{\ensuremath{\,{\rm d} x}}
\newcommand{\dy}{\ensuremath{\,{\rm d} y}}

\newcommand{\Id}{\ensuremath{\text{Id}}}

\newcommand{\id}{\mathbf{1}}
\newcommand{\tr}{\text{ tr}}
\newcommand{\halfR}{{\setR}^n_+}

\newcommand{\Scal}{{\mathcal{S}}}
\begin{document}

\title[A unified theory for some non Newtonian fluids under singular forcing]{%Existence and optimal regularity of solutions to 
A unified theory for some non Newtonian fluids under singular forcing}
\begin{abstract}
We consider a model of steady, incompressible non-Newtonian flow with neglected convective term under external forcing. Our structural assumptions allow for certain non-degenerate power-law or Carreau-type fluids. We provide the full-range theory, namely existence, optimal regularity and uniqueness of solutions, not only with respect to forcing belonging to Lebesgue spaces, but also with respect to their refinements, namely the weighted Lebesgue spaces, with weights in a respective Muckenhoupt class. The analytical highlight is derivation of existence and uniqueness theory for forcing with its regularity well-below the natural duality exponent, via estimates in weighted spaces. It is a generalization of \cite{BulDieSch15} to incompressible fluids. Moreover, two technical results, needed for our analysis, may be useful for further studies. They are: the solenoidal, weighted, biting div-curl lemma and the solenoidal Lipschitz approximations on domains.
\end{abstract}
%\begin{abstract}{We show existence of solutions in natural spaces, in case the right hand side is not in $H^{-1}$}
%\end{abstract}

%\address{Prag}
%\address{Theresienstr. 39, D-80333 Munich, Germany}

\author[M.~Bul\'{\i}\v{c}ek]{M. Bul\'{\i}\v{c}ek} %{Miroslav Bul\'{\i}\v{c}ek}
\address{Mathematical Institute, Faculty of Mathematics and Physics, Charles University in Prague
Sokolovsk\'{a} 83, 186 75 Prague, Czech Republic}
\email{mbul8060@karlin.mff.cuni.cz}

\author{J. Burczak }
\address{Institute of Mathematics, Polish Academy of Sciences, ul. \'Sniadeckich 8, 00-656 Warsaw, Poland}
\email{jb@impan.pl}

\author[S.~Schwarzacher]{S. Schwarzacher}%{Sebastian Schwarzacher}
\address{Department of Mathematical Analysis, Faculty of Mathematics and Physics,  Charles University in Prague,
Sokolovsk\'{a} 83, 186 75 Prague, Czech Republic}
\email{schwarz@karlin.mff.cuni.cz}

\thanks{M.~Bul\'{\i}\v{c}ek's work is supported by the project LL1202  financed by the Ministry of Education, Youth and Sports, Czech Republic. M.~Bul\'{\i}\v{c}ek is a member of the Ne\v{c}as Center for Mathematical Modeling. S.~Schwarzacher thanks the program PRVOUK~P47, financed by Charles University in Prague.}

\maketitle
\section{Introduction}
%\comentario{\seb{I would very much wish not to see the word quasi-linear anywhere. To my knowledge it is used for coefficient systems.}}
%\comentario{J:  No problem. I don't share your phobia :), but let's do as you prefer.  }
On a bounded domain $\Omega \subset \Rn$ with a $C^{1}$ boundary, we consider the following stationary nonlinear Stokes system
\begin{align}
\label{eq:sysA}
\begin{aligned}
  -\divv\Scal(x, \D) +\nabla p &= -\divergence
  f &&\text{ in }\Omega\\
  \divergence u&=0&&\text{ in }\Omega\\
  u&=0 &&\text{ on }\partial \Omega,
  \end{aligned}
\end{align}
where
$v :\, \Omega \to \RN$ describes the unknown velocity of the fluid, $p  \,:\, \Omega \to \setR$ describes the unknown pressure, whereas $f  \,:\, \Omega \to  \RNn$ is the given forcing and the nonlinear stress tensor is a prescribed, matrix-valued mapping $\Scal \,:\, \Omega \times \RNn \to \RNn$. We use the notation $\D =\frac{1}{2} (\nabla v + \nabla^T v)$.
% as well as forcing $f \,:\, \Omega \to \RNn$ are given.
%\comentario{J: I wonder if it makes sense to put here immediately assumptions and main theorem. I would like it, because it makes things clear from the beginning, but on the other hand, it may look heavy. What is your opinion?}

We will introduce a setting, which allows us for $f\in L^q(\Omega)$ with $q\in (1,\infty)$ to provide the full-range theory related to \eqref{eq:sysA}, namely: existence, regularity and uniqueness of its solutions (hence the eponymous `unified theory') in an arbitrary space dimension. % For the precise assumptions on $\Scal$, please consult Assumptions~\ref{ass:A} and ~\ref{ass:B}.  
Succinctly, it will suffice that $\Scal$ is monotone, linear-at-infinity (\emph{i.e.}, uniformly in $x$, $\Scal (x, \eta) \to \mu\, \eta$ as $\eta \to \infty$) and $\Scal (x, \eta) \cdot \eta$ has quadratic growth. %to have  the entire range of: existence, optimal regularity and uniqueness of solutions to \eqref{eq:sysA}. 
Typically, %also in our case 
the precise restrictions related to uniqueness will be actually slightly stronger. For detailed assumptions, we refer to Section \ref{ssec:res}.

Observe that for $\Scal (x, \eta) \cdot \eta$ growing quadratically, in case $f\not\in L^2(\Omega)$, the operator $f\mapsto \D$ related to \eqref{eq:sysA} is not anymore coupled via duality. In simpler words, $u$ can not be expected to remain an admissible test function. Therefore, the standard monotone operator theory fails. This is the analytic reason for calling such $f$'s rough forcing and the related solutions -- very weak solutions. Providing the `unified theory' for such rough forcing is the focal point of our article.

Within our assumptions, system \eqref{eq:sysA} models a steady flow of certain incompressible non-Newtonian fluids with neglected inertial forces (no convective term), that behave asymptotically Newtonian for large shear rates.
This includes famous models of incompressible non-Newtonian fluids, such as (non-degenerate) power-law fluids as well as Carreau-type fluids. For instance, we allow for  
$\Scal (x, \eta) = s (x,\abs{\eta}) \eta$ with
\begin{equation}\label{eq:ex}
\begin{aligned}
s (x,\abs{\eta})&=\mu+(\nu_0+\nu_1\abs{\eta}^2)^\frac{p-2}{2} &&\textrm{for }p\in (1,2]\text{ and } \mu>0,\, \nu_0,\nu_1\geq 0,\\
s (x,\abs{\eta})&=\min\set{\mu,(\nu_0+\nu_1\abs{\eta}^2)^\frac{p-2}{2}} &&\textrm{for }p\in (2,\infty]\text{ and } \mu>0,\, \nu_0,\nu_1\geq 0\\
\end{aligned}
\end{equation}
An important example among the substances described via stresses as above is blood, paint or ketchup.  For a discussion of the physical model see~M{\'a}lek, Rajagopal \& R{\r u}{\v z}i{\v c}ka \cite{MalR05} and M{\'a}lek \& Rajagopal \cite{MalRaj95}.
%\comentario{J: The latter reference Jaj92 is yours, I think. I couldn't locate it in lars.bib}.
 The analysis for such fluids was initiated by Ladyzhenskaya~\cite{Lad67,Lad69} and Lions~\cite{Lio69}.

In case of partial differential systems inspired by non-Newtonian flows, as our \eqref{eq:sysA}, there is no general local smoothness result of the homogenous problem, due to lack of an Uhlenbeck-type structure. This distinguishes the non-Newtonian models from, unless similar, nonlinear partial differential systems with a $p$-Laplace structure. Consequently, the nonlinear Calder\'on-Zygmund theory for non-Newtonian flows is generally not provided for $f \in L^q(\Omega)$ with large $q$'s . Therefore, the regularity theory for \eqref{eq:sysA} with high-integrable forcings is also interesting for us.

\subsection{Context and main novelties}

%\subsubsection{Linear systems}
%\comentario{J: dear Sebastian, I do not agree with your introduction/review of known result. For instance you write that \seb{ The case $q\in[2,\infty)$ follows from  interpolation and kernal representations ect. However, the case of rough data, this is when $q\in (1,2)$ is {\em always treated via duality}} for linear problems it is not true. You don't have to hinge to any kind of energy estimate (whose lack you seem as a problem), because there are kernel representations. I have below written some more details. \seb{I agree. However, I suggest not to mention coeficients at all, since we will use Non-Newtonian fluids, that behave Newtonian for larg shear rates. This is the way to sell it, or what do you think?}}
%\comentario{J: You have thrown out almost everything from part 1.1.2 S5J7. I am not fully convinced how should this part of introduction look like, but I do not like fully any of the previous versions. So let me propose another one. Now I do not distinguish between 'linear' and 'non-linear', to emphasise that our system, even if asymptotically linear, needs certain structure to work. Moreover I try to focus on rough data and do not mention much $p$-case}
Firstly, let us recall the case of the classical steady Stokes system with a rough forcing
\begin{align}
\label{eq:syslinW}
\begin{aligned}
  -\Delta v +\nabla p &= \divv f &&\text{ in }\Omega\\
  \divergence u&=0&&\text{ in }\Omega\\
  u&=0 &&\text{ on }\partial \Omega.
  \end{aligned}
\end{align}
%with $A\in\RNn$ being elliptic matrix (identity, most typically). 
%Data of our nonlinear problem \eqref{eq:sysA} correspond to  $d=0,g=0$ in \eqref{eq:syslinW},  but we write general  $d,g$ in  \eqref{eq:syslinW}, since we will formulate the respective linear result for such general inhomogenous system.
The existence of a solution $(v,p)$ to \eqref{eq:syslinW}, as well as its uniqueness and  optimal regularity
\begin{equation}\label{eq:linG}
f\in L^q(\Omega) \implies \nabla v \in L^q (\Omega), \; p 
%- \langle p \rangle
 \in \hL^q (\Omega),
\end{equation}
%with $\langle \cdot \rangle$ denoting the mean value and
for $q \in (1, \infty)$ is classical. (The circle above $L^q$ denotes null mean values and disambiguates the pressure.) The first such result is due to Cattabriga \cite{Cat61}, where the case of three space dimensions and smooth, bounded domain is considered. For further results (all space dimensions and more general domains), we refer to Borchers \& Miyakawa \cite[Section 3]{BorMiy90} and \cite{BorMiy92notes} as well as Solonnikov \cite{sol93} with their references. 

Equations and systems with a more complex structure do not allow to build such a unified theory as \eqref{eq:linG} with $q \in (1, \infty)$. Recall that even a linear, elliptic, homogenous equation can have a non-smooth solution $v$ such that $\nabla v \notin L^2$, as long as its bounded coefficients are non-smooth, see Serrin~\cite{Ser64}. This, compared with the fact that a linear, homogenous equation with bounded coefficients admits a smooth solution $v$ as long as $\nabla v \in L^2$, indicates that the case of $\nabla v \in L^q$, $q<2$ is peculiarly interesting. %And that in this case, which is strictly related to our main interest, one needs to exclude 

If the studied problem becomes nonlinear and vectorial, even smooth coefficients and smooth forcing do not assure existence of  smooth solutions, recall  \v{S}ver\'{a}k \& Yan~\cite{SvYa02} with its references. In fact, the existence or regularity theory is available only for special cases, where the nonlinearity has an appropriate structure. Its canonical examples are: monotonicity for the existence theory and the Uhlenbeck structure for regularity. It is important to observe that, up to now, both of them are insufficient to obtain existence (all the more - optimal regularity, even if the notion of optimality is clear) of solutions to problems with rough forcing, \emph{i.e.} of the type $\divergence f$ with integrability of $f$ substantially below the duality exponent dictated through the energy estimate. %Moreover, the main difficulty here is not lack of a priori estimates, but of an appropriate existence 

In this paper we develop, under suitable assumptions on the nonlinear shear stress $\Scal$, the {\em unified theory} for \eqref{eq:sysA}, namely
\begin{itemize}
\item[(i)]  Existence of its solutions, for forcing within the entire integrability range $q\in (1,\infty)$, including the difficult case of $q$'s below the duality exponent (equal $2$ within our structure).
\item[(ii)] Optimal regularity estimates and uniqueness of solutions.
\end{itemize}
The existence part and its methodology is the main novelty here. Our results generalize the ones of Bul\'i\v{c}ek, Diening \& Schwarzacher \cite{BulDieSch15} to incompressible steady flows with no inertial forces.

We find at least two of our technical results, needed to accomplish the main goal, to be interesting by themselves. These are the solenoidal, weighted, biting div-curl lemma, potentially useful for identification of limits of nonlinearities appearing in mathematical fluid dynamics, as well as our version of the solenoidal Lipschitz approximation lemma.

%Briefly, essential original points of this paper are
%\begin{itemize}
%\item[(i)] Showing existence of very weak solutions to nonlinear flows with rough data that are far beyond the natural duality exponent.
%\item[(ii)] Obtaining the optimal regularity and uniqueness of such solutions.
%\end{itemize}
%Let us now discuss each of these novelties in some more detail and provide some references to known results. 

\subsection{Main result.}\label{ssec:res}
%$\Scal$ having general quadratic structure, being uniformly elliptic, monotone and linear-at-infinity. In other words, 
For a tensor $Q\in \RNn$, its symmetrisation is denoted by $Q^s = \frac{Q + Q^T}{2}$.
We provide existence of a solution to \eqref{eq:sysA}, with $f \in L^q (\Omega)$ for all $q \in (1, \infty)$, with the related optimal regularity estimate, under the following assumptions.
%\comentario{J: I have added symmetrisation below, please check it}
\begin{assumption}
\label{ass:A}
Let $\Scal (\cdot, \cdot): \Omega \times \RNn \to \Rn$ be a Carath\'{e}odory mapping such that for positive numbers $c_0, c_1, c_2, \mu$ holds
\[
c_0\abs{Q^s}^2- c_2\leq \Scal(x,Q^s)\cdot Q, \qquad \abs{\Scal(x,Q^s)}\leq c_1\abs{Q} +c_2,\]
\[
0\leq(\Scal(x,Q^s)-\Scal(x,P^s))\cdot(Q-P),
\]
as well as it is \emph{linear-at-infinity}, \emph{i.e.}
\begin{align}
\label{eq:ass1}
\lim_{\abs{Q^s}\to\infty} \frac{\abs{\Scal(x,Q^s)- {\mu} Q^s}}{\abs{Q^s}}=0
\end{align}
for all $Q,P\in \RNn$ and uniformly in $x$.
%\comentario{J: we can restrict ourselves to symmetric tensors above and assume some symmetrization properties. do we want to? \seb{ Yes I do. Please include it. Otherwise it is not correct I think.}}
\end{assumption}
The obtained solution is unique among distributional solutions, in case one additionally has
\begin{assumption}
\label{ass:B}
Tensor $\Scal$ verifies
\[
0 < (\Scal(x,Q^s)-\Scal(x,P^s))\cdot(Q-P),
\]
and 
\begin{align}
\label{eq:ass2}
\lim_{\abs{Q^s}\to\infty}  \left|\frac{\partial \Scal (x,Q^s)}{\partial Q^s}- {\mu \Id} \, \right| = 0
  \end{align}
for all $Q^s \neq P^s \in \RNn$ and uniformly in $x$.
\end{assumption}
\begin{rem}[Admissible stress tensors]
The canonical stress tensors admissible by Assumption \ref{ass:A} are the following ones
%\comentario{some changes below}
\begin{equation}\label{prominent}
\Scal(x,\eta)=s(x,\abs{\eta})\eta, \quad \textrm{with } \; 0 \le s(x,\lambda) \le C, \quad \lim_{\lambda\to\infty}s(x,\lambda)=\mu,
\end{equation}
as long as they are monotonous.  Both Assumption \ref{ass:A} and Assumption \ref{ass:B} are satisfied by the introductory example \eqref{eq:ex}.
\end{rem}

We are ready to state our main results. The definitions of notions used in their formulations  (most of them standard) can be found in Section \ref{ssec:defs}.
%\comentario{\seb{ Here and below, I included that the mean value of the pressure is $0$. This is necessary for the uniqueness.}}
\begin{theorem}\label{th}
Let $\Scal$ of   \eqref{eq:sysA} satisfy Assumption~\ref{ass:A}. 
If $f\in L^q (\Omega)$ with $1<q<\infty$, then  \eqref{eq:sysA} admits a weak solution $(v, \pi) \in W_0^{1,q} (\Omega) \times \hL^{q} (\Omega)$. 

Moreover, for {any} $(v,\pi) \in W_0^{1,s} (\Omega)\times {  \hL^s } (\Omega)$ with a $ s>1$, solving  \eqref{eq:sysA}, the following estimate holds
\begin{equation}\label{eq:opt}
\norm{\nabla v}_{L^{q} (\Omega)} +\norm{\pi 
% -\langle \pi \rangle
}_{L^{q}(\Omega)}  \le C \left(1+ \norm{f}_{L^{q} (\Omega)}  \right).%^{ 2 -\frac{q}{2}}.
\end{equation}
The constant depends on $q$, the $C^1$-property of $\Omega$ and the quantities in Assumption~\ref{ass:A}.

If Assumption~\ref{ass:B} is additionally fulfilled, then $(v,\pi)$ solving  \eqref{eq:sysA} is unique in $W_0^{1,q} (\Omega)\times { \hL^q }(\Omega)$.%, for any $1<s$.
\end{theorem}
Theorem \ref{th} as stated is complete, since it gives at once existence, optimal integrability and uniqueness. However, for $q <2$, the $L^q$-a-priori information is not enough to develop an existence theory. To this end, we need to derive more accurate estimates, namely in weighted Lebesgue spaces. An additional benefit of this technique is that one immediately obtains the following generalisation of Theorem \ref{th} over the weighted Lebesgue spaces with the Muckenhoupt weight $A_q$.
\begin{theorem}\label{th1}
Let $\Scal$ of   \eqref{eq:sysA} satisfy Assumption~\ref{ass:A}. 

If $f\in L^q_\omega (\Omega)$ with $1<q<\infty$, $\omega\in A_q$, then \eqref{eq:sysA} admits a solution $(v, \pi) \in W_{0,\omega}^{1,q} (\Omega) \times {\hL_{\omega}^{q} } (\Omega)$. 

Moreover, for any {$(v,\pi) \in W_{0, \tilde \omega}^{1,s} (\Omega)\times \hL^s_{\tilde \omega} (\Omega)$} solving  \eqref{eq:sysA}, with an $s>1$ and  $\tilde \omega\in A_s$, the following estimate holds
\begin{equation}\label{eq:th1}
\norm{\nabla v}_{L^{q}_\omega (\Omega)} +\norm{\pi %- \langle \pi \rangle
}_{L^{q}_\omega(\Omega)}  \le C \left(1+ \norm{f}_{L^{q}_\omega (\Omega)}  \right).
\end{equation}
The constant depends on $q, A_q,$ the $C^1$-property of $\Omega$ and the quantities in Assumption~\ref{ass:A}.

If Assumption~\ref{ass:B} is additionally fulfilled, then $(v,\pi)$ solving  \eqref{eq:sysA} is unique in $W_{0, \omega}^{1,q} (\Omega)\times \hL^q_{ \omega} (\Omega)$.
\end{theorem}

Let us remark that Theorem \ref{th1} is optimal with respect to weighted spaces, since the Laplace operator is continuous in weighted Lebesgue spaces $L^q_\omega$, as long as $\omega \in A_q$, $q\in (1,\infty)$; see for instance Sawyer \cite{Saw83}, Theorem A. Observe that our Theorem \ref{th1} covers the entire range $q \in (1, \infty)$.

Let us present a short heuristics, explaining why weighted estimates are essential for an existence theory in the case of rough data. Namely, by the choice of a proper weight, the estimate \eqref{eq:th1} (utilized for a regularised problem) implies that both $\D$ and $\Scal(\cdot, \D)$ are in a weighted $L^2_\omega$ space (uniformly in a regularisation). This fact establishes a duality relation between $\D$ and $\Scal(\cdot, \D)$ which is unavailable in case of rough data within the standard Lebesgue spaces. Exploited correctly, this duality it will eventually allow to adapt a very weak version of the Minty trick. 

\begin{rem}[Measure-valued forcings are included]
 Since forcing of  \eqref{eq:sysA} is in a divergence form, we indeed cover cases of a very general forcing, for instance bounded Radon measures. Indeed, for a vector-valued bounded Radon measure $\mu$, let us solve
$-\divv \nabla h = \mu$. Hence $\nabla h \in L^r$ with any $r \in [1, \frac{n}{n-1})$, so $\nabla h = f$ is within scope of Theorems \ref{th},  \ref{th1}.
\end{rem}

%\comentario{J: I found lots of references related to measure-forced laplacian, but nothing fully satisfactory. most of them, like Droniou, \emph{Solving convection-diffusion equations with mixed, Neumann and Fourier boundary conditions and measures as data, by a duality method. Adv. Differential Equations 5 (2000), no. 10-12, 1341-1396}. or Boccardo, Gallouet, \emph{Nonlinear elliptic equations with right-hand side measures. Comm. Partial Differential Equations 17 (1992), no. 3-4, 641-655} are for a scalar equation. The idea from this Droniou (which is in fact an old idea by Stampacchia quoted there) quite evidently works for a linear system. But I am not fully happy without a direct citation and I could not access the old papers by Maz'ya  So in case you have some time, please improve this}

For the sake of completeness and to demonstrate the generality of our approach, let us finally present the respective result for systems with inhomogeneous boundary conditions  and prescribed compressibility $d$. Namely, let us consider
\begin{equation}\label{general}
  \begin{aligned}
  -\divv\Scal(x, \D)  + \nabla \pi &=-\divergence f \quad \text{ in } \Omega,\\
    \divergence v&=d \qquad \quad \; \text{ in } \Omega,\\
    \gamma(v)&=g \qquad \quad \;  \text{ on } \partial\Omega.
  \end{aligned}
  \end{equation}
	where $\gamma$ is the trace operator. In the result below, $T^q_\omega(\Omega)$ denotes the weighted trace space, see the subsection \ref{ssec:defs}. It holds
	\begin{corollary}\label{cor1}
Let $f,d\in L^q_\omega (\Omega)$, $g\in T^q_\omega(\Omega)$ %\footnote{For details on the weighted trace space $ T^q_\omega(\Omega)$ see the preliminary below.} 
with $1<q<\infty$, $\omega\in A_q$ and let $\Scal$ satisfy Assumption~\ref{ass:A}. 
Then \eqref{eq:sysA} admits a solution $(v, \pi) \in W_{\omega}^{1,q} (\Omega) \times {\hL_{\omega}^{q} }(\Omega)$. 
Moreover, if any solution of \eqref{general} for an $s>1$ enjoys  $(v,\pi) \in W^{1,s} (\Omega)\times { \hL^s }(\Omega),\gamma(v)=g$, then it satisfies
\begin{equation}
\norm{\nabla v}_{L^{q}_\omega (\Omega)} +\norm{\pi}_{L^q_\omega(\Omega)}  \le C (1+ \norm{f}_{L^q_\omega (\Omega)}+\norm{d}_{L^q_\omega (\Omega)}+\norm{ g}_{\hat T_\omega^{q}(\Omega)} ),
\end{equation}
The constant $C$ depends on $q,A_q$ the on the $C^1$-property of $\Omega$ and the quantities in Assumption~\ref{ass:A}.

If Assumption~\ref{ass:B} is additionally fulfilled, then $v, \pi$ solving  \eqref{eq:sysA} is unique in $(\gamma^{-1}(g)+W^{1,q}_0 (\Omega))\times { \hL_\omega^q }(\Omega)$. %for any $s>1$.
\end{corollary}
Finally, let us state the following remark.
\begin{rem}[A slight relaxation of assumptions]
In Assumption~\ref{ass:A} and Assumption~\ref{ass:B} the linearity-at-infinity can be relaxed. Indeed whenever it the assumptions are requested one my relax it in the following way. For every $c_0,c_1,c_2$ there exists an $\epsilon_0(c_0,c_1,c_2), m_0\in [0,\infty)$, such that \eqref{eq:ass1} can be replaced by
$
 \frac{\abs{\Scal(x,Q^s)- {\mu} Q^s}}{\abs{Q^s}}\leq \epsilon_0
$
for all $\abs{Q}\geq m_0$. 
And analogous \eqref{eq:ass2} by
$
\big|\frac{\partial \Scal (x,Q^s)}{\partial Q^s}- {\mu \Id} \, \big| \leq \epsilon_0
$
  for all $\abs{Q}\geq m_0$. 
\end{rem}

\subsection{Main technical results}

Let us gather in this section two technical results, that we would like to highlight as potentially useful in mathematical fluid dynamics. First of the is the solenoidal, weighted, biting div-curl lemma, which is a solenoidal version of Theorem 2.6 of \cite{BulDieSch15}, itself being a far generalisation of the Murat-Tartar result. 

\begin{theorem}[solenoidal, weighted, biting div--curl lemma]\label{T5}
  Let $\Omega\subset \setR^n$ denote an open, bounded set. Assume that for a given $q\in (1,\infty)$ and $\omega \in
  A_q$, there is a sequence of measurable, tensor-valued functions $a^k, s^k: \Omega \to \RNn$, $k \in \N$, such that $k$-uniformly
  \begin{equation}
\norm{a^k}_{L^q_\omega(\Omega)}   + \norm{s^k}_{L^{q'}_\omega(\Omega)}   \le C. \label{bit3}
  \end{equation}
  Furthermore, assume that for every bounded sequence
  $\{c^k\}_{k=1}^{\infty}$ in $W^{1,\infty}_{0} (\Omega)$ and every bounded solenoidal sequence   $\{d^k\}_{k=1}^{\infty}$ in $W^{1,\infty}_{0, \divv} (\Omega)$ such that
  $$
  \nabla c^k \rightharpoonup^* 0 \qquad \textrm{weakly$^*$ in }
  L^{\infty}(\Omega), \qquad   \nabla d^k \rightharpoonup^* 0 \qquad \textrm{weakly$^*$ in }
  L^{\infty}(\Omega)
  $$
one has
  \begin{align}
    \lim_{k\to \infty} \int_{\Omega} s^k \cdot \nabla d^k \dx
    &=0, \label{bit4}
    \\
    \lim_{k\to \infty} \int_{\Omega} a^k_i \partial_{x_j} c^k -
    a^k_j \partial_{x_i} c^k \dx &=0 &&\textrm{for all }
    i,j=1,\ldots,n.\label{bit5}
  \end{align}
  and that
  \begin{equation}
     \tr (a^k) \to \tr \; a \quad \text{ almost everywhere.}
    \label{bit6}
    \end{equation}
  Then, there exists a (non-relabeled) subsequence $(a^k,b^k)$  and a non-decreasing sequence of measurable subsets
  $\Omega_j\subset\Omega$ with $|\Omega \setminus \Omega_j|\to 0$ as $j\to
  \infty$ such that
  \begin{align}
  a^k &\rightharpoonup a &&\textrm{weakly in } L^1(\Omega), \label{bitfa}\\
  s^k &\rightharpoonup s &&\textrm{weakly in } L^1(\Omega), \label{bitfb}\\
  a^k \cdot s^k \omega &\rightharpoonup a \cdot s\, \omega &&\textrm{weakly in } L^1(\Omega_j) \quad \textrm{ for all } j\in \mathbb{N}. \label{bitf}
  \end{align}
\end{theorem}

The proof of Theorem \ref{T5}, presented in Section \ref{sec:techex}, relies among others on the following fine-tuning of the solenoidal Lipschitz truncations. 

\begin{theorem}[Solenoidal Lipschitz approximations on domains]
  \label{thm:liptrunc}
  Let $\Omega \subset \setR^n$ and $s>1$. Let $g \in W^{1,s}_{0,\divergence}(\Omega)$. Then for
  any $\lambda > 1 $ there exists a \emph{solenoidal Lipschitz truncation} $g^\lambda \in
  W^{1,\infty}_{\divergence}(\Omega)$ such that
  \begin{alignat}{2}
    \label{eq:lip1}
    g^\lambda &= g \quad \text{and} \quad \nabla g^\lambda = \nabla g
    &\qquad&\text{in $\set{M(\nabla g)\leq \lambda}\cap\Omega$},
    \\
    \label{eq:lip2}
    \abs{\nabla g^\lambda} &\leq \abs{\nabla g}\chi_{\set{M(\nabla
        g)\leq\lambda}}+C\, \lambda
    \chi_{\set{M(\nabla g)>\lambda}} &&\textrm{almost
      everywhere}.
  \end{alignat}
  Further, if $\nabla g\in L^{p}_\omega(\Omega)$ for
  some $1\leq p<\infty$ and $\omega \in {A}_p$, then
  \begin{align}\label{itm:weight2}
    \begin{aligned}
      \int_{\Omega}\abs{\nabla g^\lambda}^p \omega \dx &\leq
      C\int_{\Omega} \abs{\nabla g}^p \omega \dx,
      \\
      \int_{\Omega} \abs{\nabla (g-g^\lambda)}^p \omega \dx&\leq
      C \int_{\Omega\cap \set{M(\nabla g)>\lambda}} \abs{\nabla g}^p \omega \dx,
    \end{aligned}
  \end{align}
where the constant $C$ depends on $({A}_p(\Omega),\Omega, N, p)$.
\end{theorem}
Proof of Theorem \ref{thm:liptrunc} can be found in Section \ref{sec:techex}.

\subsection{Further research}

Let us point out the significance of our results for future research, particularly,  the flexibility of the developed existence scheme.

Firstly, consider the full Navier-Stokes analogue of \eqref{eq:sysA}. It involves an additional convective term. However, in three dimensions, it is possible to treat it as a right hand side with respect to a-priori estimates and as a compact perturbation  with respect to the existence analysis. This will be presented in our future work. For results on existence of solutions to steady non-Newtonian Navier-Stokes flows with non-rough forcing, see Diening, M\'alek \& Steinhauer \cite{DieMalSte08} and Bul\'i\v{c}ek, M\'alek, Gwiazda \& \'Swierczewska-Gwiazda \cite{BulMalGwiaSwi09}.

Another generalization is related to considering degeneracies, for instance, the degenerate power-law model $\Scal(x,Q)=\nu\abs{Q}^{p-2}Q$. Recently, it was possible to establish an existence theory for the related $p$-Laplace system, see Bul\'i\v{c}ek \& Schwarzacher~\cite{BulSch16}. Even though it holds only for exponents $q$ being close to the natural exponent $p$, it is the first existence proof for degenerate systems below the duality exponent. A generalization to degenerate fluids seems achievable. It would also match into the regularity
 theory available for the  degenerate Stokes systems, compare Diening \& Kaplick\'y~\cite{DieKap12} and Diening, Kaplick\'y \& Schwarzacher \cite{DieKapSch13}.

Finally, we wish to emphasize that the very weak weighed duality relation discovered here has a considerable potential for numerical schemes and their analysis.

\section{Preliminaries}
\subsection{Structure of the paper}\label{S2}
This section gathers certain auxiliary tools for the proofs. Section \ref{sec:reg} presents an \emph{a-priori} type estimate: in Theorem \ref{thm:a-priori} there, we provide quantitative regularity estimate \eqref{eq:opt}, under an additional assumption that solution to \eqref{eq:sysA} belongs to a certain $L^s (\Omega)$ regularity class, $s>1$. This result relies on a regularity theory for weighted linear Stokes, that we partially needed to provide.
Section \ref{sec:techex} contains proofs of the main technical results, namely of Theorem \ref{T5} and of Theorem \ref{thm:liptrunc}.
Finally, Section \ref{sec:ex} provides proofs of our main theorems, presented in Section \ref{ssec:res}.

\subsection{Basic notation and definitions}\label{ssec:defs}
\subsubsection{Function spaces}
For $p\in [1,\infty)$ and $\omega$ being a weight \emph{i.e.} a measurable function that is almost everywhere finite and positive, let us define the weighted Lebesgue space $L^p_\omega(\Omega)$ and its norm $\norm{\cdot}_{L^p_\omega}$ as
$$
L^p_{\omega}(\Omega):=\biggset{f:\Omega \to \setR^n; \; \text{measurable},\;  \norm{f}_{L^p_\omega}
:= \bigg(\int_{\Omega} |u(x)|^p\omega(x)\dx\bigg)^{\frac 1p} <\infty}.
$$
The space ${\mathring L^p_{\omega}(\Omega) }$ contains all functions $f\in L^p_\omega(\Omega)$ with $\int_\Omega f\, dx=0$.
%\comentario{J: as you can see, finally I decided for a circle above lebesgue. hat is used in for a homogenous Sobolev. so hat for lebesgue would make things a little messed up. I could live with that, but maybe it is better to be  clear }
The weighted Sobolev space $W^{1,p}_\omega(\Omega)$ consists of all functions where both the distributional derivative $\nabla f$ and $f$ are in $L^p_\omega(\Omega)$. 

The homogeneous Sobolev space $\hat {W}^{1,p}_\omega(\Omega)$  is the space of all functions such that $\nabla f\in L^p_\omega(\Omega)$ (and $f$ belongs to the natural embedded space; $\hat {W}^{1,p}_\omega(\Omega) \neq {W}^{1,p}_\omega(\Omega)$ only in unbounded domains). 

Since weights may have a certain impact on the exact shape of trace space, one typically defines it only semi-explicitly as $\gamma ({W}^{1,p}_\omega(\Omega) {\cap W^{1,1}(\Omega)})$, where $\gamma: {W}^{1,1}(\Omega) \to {L}^{1}(\partial \Omega) $ is the canonical trace operator. In case of an unbounded domain, one additionally localizes the domain by an intersection with a ball.  For some more details, compare Fr\"ohlich \cite{Fro07} and \cite{Fro03}, section 3.3 with their references. The zero trace subspaces of   ${W}^{1,p}_{\omega}(\Omega)$ and  $\hat {W}^{1,p}_{\omega}(\Omega)$ 
are denoted by  ${W}^{1,p}_{0,\omega}(\Omega)$ and ${\hat W}^{1,p}_{0,\omega}(\Omega)$, respectively. For brevity, we will write $ T^{q}_\omega(U)$ for  $\gamma ({W}^{1,p}_\omega(\Omega)\cap {W^{1,1}(\Omega)})$ and $\hat T^{q}_\omega(U)$ for  $\gamma (\hat {W}^{1,p}_\omega(\Omega)\cap {W^{1,1}_{\text{loc}}(\overline{\Omega}))}$.

All the mentioned spaces are Banach spaces. {In the case considered here, namely} the case of Muckenhoupt weights $\omega \in A_p$ and $p\in (1,\infty)$, the above defined spaces are additionally reflexive and separable. These and more properties are discussed in Stein~\cite[Chapter 3]{Ste93}, for instance. Moreover, by \eqref{eq:lqprop} below, we find in case of Muckenhoupt weights $\omega \in A_p$ that $W^{1,p}_\omega(\Omega)\subset W^{1,1}(\Omega)$ and $\hat{W}^{1,p}_\omega(\Omega)\subset W^{1,1}_{\text{loc}}(\overline{\Omega})$,
hence functions that are bounded in $W^{1,p}_\omega(\Omega)$, $\hat{W}^{1,p}_\omega(\Omega)$ 
possess weak derivatives and well-defined traces.
 
Finally,  $W_{0, \divv,\omega}^{1,q} (\Omega)$ is defined as the closure of $C^\infty_{0,\divergence}(\Omega)$ (the smooth, compactly supported and solenoidal functions) with respect to the $W^{1,q}_{\omega}$-norm.

For any vector- or tensor-valued $f\in L^1_{\loc}(\setR^n)$ we define its Hardy-Littlewood maximal function $Mf$ in a standard manner as follows
\[
Mf(x):=\sup_{R>0}\;\dashint_{B_R(x)}\abs{f(y)}\dy, % \quad \textrm{with} \quad \dashint_{B_R(x)}\abs{f(y)}\dy:=\frac{1}{|B_R(x)|}\int_{B_R(x)}\abs{f(y)}\dy,
\]
where $B_R(x)$ denotes a ball with radius $R$ centered at $x\in \setR^n$. 

\subsubsection{A notion of solution}
Let us introduce the standard 
\begin{definition}[Distributional solution]
A couple $(v, p) \in {W}^{1,1}_{0,\divergence}(\Omega) \times  {L}^{1}( \Omega)$ is a distributional solution to \eqref{eq:sysA} iff for any $\phi \in C^\infty_0 (\Omega)$ holds
\[
\begin{aligned}
\int_\Omega \Scal(x, \D) \nabla \phi - p \, \divv \phi  &= \int_\Omega f \nabla \phi, \\
 \gamma (u) & = 0.
 \end{aligned}
\]
\end{definition}
%\comentario{\seb{Did you define the trace operator $\gamma??$ do it or do not use it :-). I think divergence free functions have traces just like all the others :-).}}
An analogous definition, with natural modifications, will be used for the inhomogenous problem.

In the following, we will sometimes call a $(v,p)$ a weak solution, provided it belongs to the optimal regularity class (with respect to regularity of $f$). 
\subsection{An algebraic lemma}
Let us begin with an algebraic Lemma, which can be found as Lemma 4.1 in  Buli\v{c}ek, Diening \& Schwarzacher  \cite{BulDieSch15}.
\begin{lemma}\label{L:algebra}
Let $\Scal$ fulfill Assumptions \ref{ass:A}, \ref{ass:B}. Then for every  $\delta>0$ there exists $C$ such that for all $x\in \Omega$ and all $Q, P \in \mathbb{R}^{n\times n}$ there holds
\begin{equation}\label{algebra}
|\Scal(x,Q)-\Scal(x,P) -%\tilde{\Scal}(x)
\mu (Q-P)|\le \delta |Q-P| + C(\delta).
\end{equation}
\end{lemma}

%\input{Tools}

%The primary intention of this section is to introduce the standard notations
%and standard tools mainly in harmonic analysis and measure theory that
%we shall used in the paper. Although, we feel that all of them are
%well known or can be deduced from the standard results, we sketch  the
%proofs of those we were not able to find in the existing literature.

\subsection{Muckenhoupt weights}\label{sec:muck} To provide optimal regularity and to mimic the $L^2$ duality, we resort to  $L_\omega^2$ with a weight $\omega$
from  the Muckenhoupt class. 
\begin{definition}
For  $p\in [1,\infty)$, we say that a weight $\omega$
belongs to the \emph{Muckenhoupt class $A_p$} if and only if
there exists a positive constant $A$ such that for every ball $B
\subset \setR^k$ holds
\begin{alignat}{2}
  \label{defAp2}
  \left(\dashint_B
    \omega\dx\right) \left(\dashint_B\omega^{-(p'-1)}\dx\right)^\frac{1}{p'-1}
  &\le A & \qquad\qquad&\text{if $p \in (1,\infty)$},
  \\
  \label{defA1}
  M\omega(x)&\le A\, \omega(x) &&\text{if $p=1$}.
\end{alignat}
We denote by
$A_p(\omega)$ the smallest constant $A$ for which the
inequality~\eqref{defAp2}, resp.~\eqref{defA1}, holds. 
\end{definition}
\subsubsection{Basic properties}
For $1 \le p \le q < \infty$ holds $A_p \subset A_q$. 
The maximum $\omega_1 \vee \omega_2$
and minimum $\omega_1 \wedge \omega_2$ of two $A_p$-weights is again
an $A_p$-weight. For $p=2$, since $\frac1{\omega_1
      \wedge \omega_2} \le \frac1{\omega_1} + \frac1{\omega_2}$ almost everywhere, we have straightforwardly
\begin{equation}
  \label{eq:A2min}
    \dashint_B (\omega_1 \wedge \omega_2)\dx \; \dashint_B \frac1{\omega_1
      \wedge \omega_2}\dx \leq A_2(\omega_1) + A_2(\omega_2).
\end{equation}

For $\omega \in A_q$, $q \in(1, \infty)$ we will write $\omega'= \omega^{- \frac{1}{q-1}}$. H\"older inequality gives $\omega \in A_q \iff \omega' \in  A_{q'}$.

\subsubsection{Relation to the maximal function}
Due to the celebrated result of Muckenhoupt \cite{Mu72}, we know
that $\omega \in A_p$ for $1<p<\infty$ is equivalent to the
existence of a constant $A'$, such that  for all $f\in L^p_\omega(\setR^n)$
\begin{equation}\label{defAp}
\int\abs{Mf}^p\omega\dx\leq A'\,\int\abs{f}^p\omega\dx.
\end{equation}
Another link between the maximal function and  $A_p$-weights is given by
\begin{lemma}\label{cor:dual}
Let $f \in L^1_{\loc}(\setR^n)$ be such that $Mf<\infty$ almost everywhere in $\setR^n$. Then for all $\alpha \in (0,1)$ we have $(Mf)^{\alpha} \in A_1$. Furthermore, for all $p\in (1,\infty)$ and all $\alpha\in (0,1)$ there holds
$(Mf)^{-\alpha(p-1)}\in A_p$.
\end{lemma}
For proof, see pages 229--230 in Torchinsky \cite{86:_real} and page 5 in Turesson \cite{Tur00}.

Consequently, 
\begin{equation}\label{weight:s}
g\in L^s(\Omega) \;\text{ for an }\; s \in (1,2) \implies g\in L^2_{\omega_1}(\Omega) \; \text{ with } \; \omega_1=(Mg)^{s-2}\in A_2,
\end{equation}
because\footnote{Here and in what follows, when we deal with maximal function and a function defined on a domain, we extend the function over the full space by $0$.}
\begin{align}
\label{weight1}
\int g^2(Mg)^{s-2}\dx\leq \int g^s\dx \leq \int (Mg)^2(Mg)^{s-2}\dx\leq A' \int g^2(Mg)^{s-2}\dx.
\end{align}
% we find for $q\in (1,2)$, that $g\in L^q_{\omega_1}$, for $\omega_1=(Mg)^{s-q}$. Since $(s-q)\in (1-q,1)$, and therefore, if $s<q$
%\begin{align}
%\label{weight1}
%\int g^q(Mg)^{s-q}\dx\leq \int g^s\dx \leq \int (Mg)^q(Mg)^{s-q}\dx\leq c\int g^q(Mg)^{s-q}\dx.
%\end{align}
%And if $s>q$, then
%\begin{align}
%\label{weight2}
%\int g^s\dx\leq \int g^q(Mg)^{s-q}\dx\leq \int (Mg)^{q}\dx\leq c\int g^s\dx.
%\end{align}
Finally, we will need also that for every $p\in (1,\infty)$ and $\omega\in A_p$, there exists an $s\in (1,\infty)$ depending only on $A_p(\omega)$, such that $L^p_\omega (\Omega)
\embedding L^s_{\loc}(\Omega)$. Moreover, the related inequality
\begin{equation}
  \label{eq:lqprop}
    \bigg( \dashint_B \abs{f}^s\dx\bigg)^{\frac 1s} \leq C(A_p(\omega)) \left(\dashint_B
    \omega\dx\right)^{\frac 1p}   \bigg(
  \int_B \abs{f}^p \omega\dx\bigg)^{\frac 1p},
%  \end{aligned}
\end{equation}
holds. See formula (3.5) from \cite{BulDieSch15}.
\subsubsection{A miracle of extrapolation}
The seminal work by Rub\'io de Francia \cite{Rub84} implies that if an operator is bounded between $L^{p_0}_\omega$ for a $p_0 \in (1, \infty)$ and every $\omega \in A_{p_0}$, then it is also bounded for $L^{p}_\omega$ for every $p \in (1, \infty)$ and every $\omega \in A_{p}$. We will refer to this fact as a `miracle of extrapolation', compare Theorem~1.4 of monograph \cite{CruMarPerBook} by Cruz-Uribe, Martell \& P\'erez and its references.

\subsection{Very weak compactness}
Since we couldn't locate the reference, we provide proof of the following very weak compactness result.
\begin{lemma}
\label{lem:dualcomp}
For $\omega\in A_q$ it holds $L^{q}_{\omega} \hookrightarrow (W^{1,q'}_{\omega', 0})^*$, with the embedding being (sequentially) compact.
\end{lemma}
\begin{proof}
Let us pick a uniformly bounded sequence $g_j\in L^q_{\omega}(\Omega)$, $\norm{g_j}_{L^q_\omega (\Omega)}\leq c$. Since $L^q_\omega(\Omega)$ is reflexive, the weak compactness implies that on a subsequence $g_j\weakto g$. By subtracting the limit, we may assume with no loss of generality that $g\equiv 0$. By the dual norm definition, we find $\psi_i\in W^{1,q'}_{\omega', 0} (\Omega)$, such that
$\norm{\psi_j}_{W^{1,q'}_{\omega', 0} (\Omega)}=1$ and $\norm{g_j}_{(W^{1,q'}_{\omega', 0} (\Omega))^*}\leq 2\skp{g_j}{\psi_j}$. Moreover, we find by Theorem~2.3, \cite{Fro07} a convergent (non-relabeled) subsequence $\psi_{j}\to \psi$, in $L^{q'}_{\omega'}(\Omega)$. This implies, by the weak-strong-coupling, that $\norm{g_j}_{(W^{1,q'}_{\omega', 0} (\Omega))^*} \leq 2\skp{g_j}{\psi_j }\to 0$ on a subsequence, which is the (sequential) compactness of our embedding.
\end{proof}

\subsection{Convergence tools}
In order  to identify the limit correctly, we will use
\begin{lemma}\label{thm:blem}
   Let $\Omega$ be a bounded domain in $\setR^n$ and let
   $\{v^k\}_{n=1}^{\infty}$ be a bounded sequence in $L^1(\Omega)$. Then
   there exists a non-decreasing  sequence of measurable subsets
   $\Omega_j\subset\Omega$ with $|\Omega \setminus \Omega_j|\to 0$ as $j\to \infty$ such that for every $j\in \mathbb{N}$ and every $\varepsilon>0$ there exists a $\delta>0$ such that for all $A\subset \Omega_j$ with $\abs{A}\leq \delta$ and all $n\in \mathbb{N}$ the following holds
\begin{equation}
\int_A\abs{v^k}\dx\leq \epsilon.
\end{equation}
 \end{lemma}
 The Chacon's Biting Lemma from Ball \& Murat \cite{BallMurat:89} has as its thesis weak-$L^1$ precompactness, which implies thesis of Lemma \ref{thm:blem} in view of Vitali's Theorem.

\section{Regularity estimate}\label{sec:reg}
The main result of this section is Theorem \ref{thm:a-priori} below. It shows that any distributional solution $(v, \pi)$ to \eqref{eq:sysA} enjoys optimal regularity estimate, provided additionally $\nabla v, \pi \in L^s (\Omega)$, for an $s>1$. The relation between $q$, $\omega$, right hand side $f$ and $s$ will become clear in the next section.
\begin{theorem}
\label{thm:a-priori}
Let $\Omega$ be a bounded domain with $\partial \Omega \in  \mathcal{C}^{1}$ and $\Scal$ satisfy Assumption~\ref{ass:A}.
Let $f\in L^q_\omega (\Omega)$, with $1<q<\infty$, $\omega\in A_q$, $s\in (1,\infty)$ and $\Scal$ satisfy Assumption~\ref{ass:A}. Then any distributional solution of \eqref{eq:sysA} that enjoys additionally  $(v,\pi) \in W_{0}^{1,s} (\Omega)\times L^s(\Omega)$, satisfies
\begin{equation}
\norm{\nabla v}_{L^{q}_\omega (\Omega)} +\norm{\pi - \langle \pi \rangle}_{L^q_\omega(\Omega)}  \le C \left(1+ \norm{f}_{L^q_\omega (\Omega)}\right).
\end{equation}
Analogously for the inhomogenous case: if $d\in L^q_\omega(\Omega)$ and $g\in T^q_\omega(\Omega)$, then any distributional  solution of \eqref{general} that enjoys additionally  $(v,\pi) \in W^{1,s} (\Omega)\times L^s(\Omega),\gamma(u)=g$ satisfies
\begin{equation}
\norm{\nabla v}_{L^{q}_\omega (\Omega)} +\norm{\pi - \langle \pi \rangle}_{L^q_\omega(\Omega)}  \le C \left(1+ \norm{f}_{L^q_\omega (\Omega)}+\norm{d}_{L^q_\omega (\Omega)}+\norm{ g}_{\hat T_\omega^{q}(\Omega)} \right),
\end{equation}
where all constants depend only on Assumption~\ref{ass:A}, $A_q(\omega),q$ and on the modulus of continuity of $\partial\Omega$.
\end{theorem}
Proof of Theorem \ref{thm:a-priori} occupies the end of this section. As the main ingredient of its proof, we need
\subsection{$L_\omega^q$-theory for linear Stokes}
%The starting point for getting crucial estimates of this paper is the following result on maximal $L^q_\omega$-regularity for the linear, stationary Stokes system. 

\begin{lemma}
  \label{lem:CZ}
  Let $\Omega$ be a bounded domain with $\partial \Omega \in  \mathcal{C}^{1}$ and  $(w,p) \in W_\omega^{1,q}(\Omega) \times L_\omega^{q}(\Omega)$ be a distributional solution to
  \begin{equation}\label{eq:linSt}
  \begin{aligned}
    -\divergence (\Dw) + \nabla p &=-\divergence F \quad \text{ in } \Omega,\\
    \divergence w&=d \qquad \quad \; \text{ in } \Omega,\\
    \gamma(w)&=g \qquad \quad \;  \text{ on } \partial\Omega.
  \end{aligned}
  \end{equation}
 Then for any $F,d\in L^q_\omega(\Omega)$ and $g\in T^q_\omega(\Omega)$ with $q\in (1,\infty)$ and $\omega\in A_q$  \begin{equation}\label{aoprioriconst}
    \norm{ w}_{W_\omega^{1,q}(\Omega)}   +  \norm{p - \langle p \rangle}_{L_\omega^{q}(\Omega)} \leq C \left( \norm{F}_{L_\omega^{q}(\Omega)} + \norm{d}_{L_\omega^{q}(\Omega)} + \norm{ g}_{\hat T_\omega^{q}(\Omega)} \right),
  \end{equation}
where $C = C(q, A_q(\omega), \partial\Omega)$. %constant the shape of the boundary as well as the dimension, that if $(\nabla w,p)\in L^q_\omega(\Omega)$ , then
\end{lemma}
We couldn't find the exact reference concerning the case of a bounded domain, so we provide the proof.
\begin{proof}
Let $U$ be either the full space $\setR^n$ or the half-space $\halfR$. Recall that by $\hat W^{1,q}_\omega(U)$ we denote the homogeneous Sobolev space. In view of Theorems 5.1 and 5.2 by Fr\"ohlich \cite{Fro03} (cases of $\setR^n$ and $\halfR$, respectively),  for every  $f \in (\hat W^{1,q'}_{0, \omega'}(U))^*$,   $d\in L^q_\omega(U)$ and $g\in \hat T^{q}_\omega(U)$, the following problem %on half-space $\Omega = \halfR$
  \begin{equation}\label{eq:linStF}
  \begin{aligned}
    -\divergence ( \Dw) + \nabla p &=f \qquad \text{ in } U,\\
    -\divergence w&=d \qquad  \text{ in }  U,\\
    (\gamma(w)&=g \qquad \text{ on } \setR^{n-1} \quad \text{ in case of U }= \halfR),
  \end{aligned}
  \end{equation}
  admits a unique  weak solution $(w, p) \in \hat W_{ \omega}^{1,q} (U) \times L^{q}_\omega(U) $ that enjoys the estimate
\begin{equation}\label{aoprioriconstFp}
    \norm{\nabla w}_{L_\omega^{q}(U)}   +  \norm{p}_{L_\omega^{q}(U)} \leq C \left(\norm{f}_{(\hat W^{1,q'}_{0, \omega'} (U))^*}+\norm{d}_{L_\omega^{q}(U)}+\norm{g}_{\hat T_\omega^{q}(U)}\right),
  \end{equation}
where the term involving $g$ naturally appears only for the half-space, $C = C(q, A_q, \Omega)$, $q\in (1,\infty)$ and $\omega\in A_q$. In particular, for $f = -\divv F$, where $F \in L^{q}_\omega(U)$
\begin{equation}\label{aoprioriconstF}
    \norm{\nabla w}_{L_\omega^{q}(U)}   +  \norm{p}_{L_\omega^{q}(U)} \leq C \left(\norm{F}_{L^{q}_\omega(U)}+\norm{d}_{L_\omega^{q}(U)}+\norm{ g}_{\hat T_\omega^{q}(U)}\right),
  \end{equation}
so our thesis follows\footnote{ In fact, the cited results from  \cite{Fro03} consider $\nabla w$ in place of $\Dw$ in the problem formulation. The same result holds for \eqref{eq:linStF} by a redefinition of the pressure as $p- \divv w$.  %Analogously, this observation allows us to obtain in what follows results for $\Dw$ based on analysis for $\nabla w$ also for the problem with zero-traces on a bounded domain.
}.

 Hence to finish our proof, we are left with performing the last step: from full-space and half-space to a bounded domain. Unluckily, the available results (see Fr\"ohlich \cite{Fro07} or Schumacher \cite{Sch08}) do not cover the needed case of weak forcing $\divv F$, $F \in L^q_\omega$, so let us provide some details of this last step.

% (we use the case of $div u = k$, since we will use the localization in the next step that may produce lack of incompressibility) . For the estimates in domain, we perform the following steps: 
Recall that our goal here is merely the optimal regularity and not existence. Hence in what follows, we assume to have a distributional solution of a considered problem and we aim at showing \eqref{aoprioriconstF} for that $w$.

Firstly, let us consider a distributional solution to problem \eqref{eq:linStF} on $\Omega = E$ being a bended half space with a small bend and with $g=0$. By a small bend we mean that there exists smooth $\Sigma:\halfR  \ni \tilde x \to E \ni x$ having the form $\Sigma (\tilde x_1, \dots \tilde x_{n-1}, \tilde x_n) = (\tilde x_1, \dots \tilde x_{n-1}, \tilde x_n + \sigma (\tilde x_1, \dots \tilde x_{n-1})  )$ with small derivatives of $\sigma$ ($\Sigma$ being a small perturbation of identity). %and $\nabla^2 \Sigma$. 
Distributional formulation of  \eqref{eq:linStF} 
  \begin{equation*}%\label{eq:linStF}
  \begin{aligned}
\int_E  w_{x_j}^i \phi_{x_j}^i  - p \phi_{x_i}^i  &=\int_E  f^{i} \phi^i  \\
\int_E  w^i \psi_{x_i} &= \int_E d \psi
  \end{aligned}
\end{equation*}
translates for new functions $w \circ \Sigma = \tilde w$ etc. via a straightforward computation, with an observation that change of variables is volume-preserving, into
% (no summation over repeated $n$'s)
  \begin{align*}
&\int_{\halfR}  (\tilde w_{\tilde x_j}^i - \sigma_{\tilde x_j} \tilde w^i_{\tilde x_n} )  (\tilde \phi_{\tilde x_j}^i - \sigma_{\tilde x_j} \tilde \phi^i_{\tilde x_n} ) \id_{ \{j < n \}}  +  \tilde w_{\tilde x_n}^i \tilde \phi_{\tilde x_n}^i  - \tilde p \big( (\tilde \phi_{\tilde x_i}^i - \sigma_{\tilde x_i} \tilde \phi^i_{\tilde x_n} ) \id_{ \{i < n \}} + \tilde \phi_{\tilde x_n}^n  \big)  \\
&= \int_{\halfR}  \tilde f^{i} \tilde \phi^i, \\
& \int_{\halfR} \tilde w^i (\tilde \psi_{\tilde x_i} - \sigma_{\tilde x_i} \tilde \psi_{\tilde x_n} ) \id_{ \{i < n \}}  +  \tilde w^n \tilde \psi_{\tilde x_n}  = \int_{\halfR}  \tilde d \tilde \psi,
\end{align*}
(no summation over repeated $n$'s)
\emph{i.e.}, after reordering
  \begin{equation*}%\label{eq:linStF}
  \begin{aligned}
& \int_{\halfR}  \tilde w_{\tilde x_j}^i  \tilde \phi_{\tilde x_j}^i  - \tilde p \tilde \phi_{\tilde x_i}^i   =\int_{\halfR} \sigma_{\tilde x_j} \tilde w^i_{\tilde x_n} \tilde \phi_{\tilde x_j}^i \id_{ \{j < n \}}  +  ( \tilde w_{\tilde x_j}^i  \sigma_{\tilde x_j} - \tilde w_{\tilde x_n}^i  \sigma^2_{\tilde x_j} - \tilde p  \sigma_{\tilde x_i}  ) \id_{ \{j < n \}}  \tilde \phi^i_{\tilde x_n} +  \tilde f^{i} \tilde \phi^i, \\
& \int_{\halfR} \tilde w^i \tilde \psi_{\tilde x_i} = \int_{\halfR} \big( \tilde d  -(\tilde w^i \sigma_{\tilde x_i} \id_{ \{i < n \}})_{\tilde x_n} \big)  \tilde \psi.
  \end{aligned}
\end{equation*}
This shows that $(\tilde w, \tilde p)$ solves distributionally
  \begin{equation*}%\label{eq:linStF}
  \begin{aligned}
    -\divergence (\nabla \tilde w) + \nabla \tilde p &=-\divergence B + \tilde f   &&\text{ in } \halfR,\\
    -\divergence \tilde w&= D &&\text{ in }  \halfR,\\
 \gamma(   \tilde w)&= 0  &&\text{ on } \setR^{n-1} 
  \end{aligned}
  \end{equation*}
  
with
\[
B^{ij} = \begin{cases}   \sigma_{\tilde x_j} \tilde w^i_{\tilde x_n} &\text{ for } j<n,\\
 ( \tilde w_{\tilde x_j}^i  \sigma_{\tilde x_j} - \tilde w_{\tilde x_n}^i  \sigma^2_{\tilde x_j} - \tilde p  \sigma_{\tilde x_i}   ) \id_{ \{j < n \}}  \quad &\text{ for } j=n,
\end{cases}
\]
and 
\[
D = \tilde d  -(\tilde w^i \sigma_{\tilde x_i} \id_{ \{i < n \}})_{\tilde x_n}.
\]

Hence we can use \eqref{aoprioriconstF} on $\Omega = \halfR$ for  $(\tilde w, \tilde p)$ and data $-\divv B + \tilde f, D$. It gives, after taking into account the form of $B, D$
\begin{multline*}%\label{aoprioriconstF}
    \norm{\nabla \tilde w}_{L_\omega^{q}( \halfR)}   +  \norm{\tilde p}_{L_\omega^{q}( \halfR)} \leq  c \norm{\tilde f}_{(\hat W^{1,q'}_{0, \omega'} (\halfR))^*} +\\
     c(|\nabla \sigma|_\infty + |\nabla \sigma|^2_\infty)    ( \norm{\nabla \tilde w}_{L_\omega^{q}( \halfR)}   +  \norm{\tilde p}_{L_\omega^{q}( \halfR)} ) + c\norm{\tilde d}_{L_\omega^{q}( \halfR)} + |\nabla^2 \sigma|^2_\infty   \norm{ \tilde w}_{L_\omega^{q}( \halfR)}.
  \end{multline*}
% \comentario{ZZZ: pay attention to traces}
Smallness of the bend, \emph{i.e.} of derivatives of $\sigma$ in relation to $c =  c(q, A_q, \halfR)$ implies then
\begin{equation}\label{aoprioriconstF2}
    \norm{\nabla \tilde w}_{L_\omega^{q}( \halfR)}   +  \norm{\tilde p}_{L_\omega^{q}( \halfR)} \leq  c \norm{\tilde f}_{(\hat W^{1,q'}_{0, \omega'} (\halfR))^*}  + c\norm{\tilde d}_{L_\omega^{q}( \halfR)} + \delta \norm{ \tilde w}_{L_\omega^{q}( \halfR)}
  \end{equation}
Since $\Sigma$ is a small, volume preserving  perturbation of identity, \eqref{aoprioriconstF2} gives  for  $( w, p)$  solving problem \eqref{eq:linStF} with data $ f, d, g=0$ on $\Omega = E$ being a bended half space with a small bend 
\begin{equation}\label{aoprioriconstF3}
    \norm{\nabla  w}_{L_\omega^{q}(E)}   +  \norm{p}_{L_\omega^{q}(E)} \leq c ( \norm{f}_{(\hat W^{1,q'}_{0, \omega'} (E))^*}  + \norm{d}_{L_\omega^{q}(E)} +  \norm{w}_{L_\omega^{q}(E)}).
  \end{equation}
Next, let us consider a distributional solution  $(w, p)$ to our target problem \eqref{eq:linSt}, still with $g=0$. For a cutoff function $\eta$ with a small support and an arbitrary test functions: vector valued $\phi$ and scalar valued $\psi$, we have that
  \begin{equation*}%\label{eq:linStF}
  \begin{aligned}
\int_V  w_{x_j}^i (\phi^i \eta)_{x_j}  - p (\phi^i \eta)_{x_i}  &=\int_V  F^{ij} (\phi^i \eta)_{x_j}  \\
\int_V  w^i (\psi \eta)_{x_i} &= \int_V d \psi \eta
  \end{aligned}
\end{equation*}
for $V$ being either the little-banded half space $E$, when we localize near the boundary $\partial \Omega$ or $\setR^n$, when we localize away from the boundary. Observe that due to our assumption that $\partial \Omega \in C^{{1}}$ and $\partial \Omega$ is compact, we can always find a small absolute number $\delta$, such that the intersection $B_\delta\cap\partial\Omega$, can be described with local coordinates $\sigma$, such that $\norm{\nabla\sigma}_\infty$ is conveniently small for any $B_\delta\subset \setR^n$. We introduce a partition of unity $\eta^k$ on $\Omega$, where $\eta^k$ have support on a $B_\delta$ and a number $c$.  
The localized functions $\bar w = w \eta^k$, $\bar p = p \eta^k - c$ satisfy distributionally
%\comentario{Is $V$ introduced somewhere?}
  \begin{equation*}%\label{eq:linStF}
  \begin{aligned}
    -\divergence (\nabla \bar w) + \nabla (\bar p) &=h -\divv H \qquad \text{ in } V,\\
    -\divergence \bar w&= w^i \eta^k_{x_i} + \bar d  \qquad \text{ in }  V,\\
 (\bar w&=0 \qquad \text{ on } \partial E \quad \text{ in case of } V = E),
  \end{aligned}
  \end{equation*}
  where
  \[
h^i = F^{ij} \eta_{x_j} + w^i_{x_j} \eta_{x_j}^k + p \eta_{x_i}^k , \qquad   H^{ij} = w^i \eta_{x_j}^k + F^{ij} \eta^k.
  \]
Estimates  \eqref{aoprioriconstFp},  \eqref{aoprioriconstF3} give
 \begin{multline*}
 \norm{ \nabla \bar w}_{L_\omega^{q} (V)} +  \norm{ \bar p}_{L_\omega^{q} (V)} \le \\
 C (\eta^k, q, A_q, \Omega) \left( \norm{F}_{L_\omega^{q} (\Omega)} + \norm{d}_{L_\omega^{q} (\Omega)} +  \norm{ w}_{L_\omega^{q} (\Omega)}  +  \norm{\nabla w}_{(\hat W^{1,q'}_{0, \omega'} (\Omega))^*}     + \norm{p}_{(\hat W^{1,q'}_{0, \omega'} (\Omega))^*}   \right).
 \end{multline*}
 Hence, summing {over $k$} and using weighted Poincar\'e inequality (see Theorem 2.3 of \cite{Fro07}) and choosing  $c = \langle p \rangle_\Omega$, we arrive at
 \begin{equation}
  \begin{aligned}
  \label{eq:a-prior1}
& \norm{ w}_{W_\omega^{1,q} (\Omega)} +  \norm{  p-\langle p \rangle}_{L_\omega^{q} (\Omega)} \\
& \le  C (q, A_q, \Omega) \left( \norm{F}_{L_\omega^{q} (\Omega)} +  \norm{ w}_{L_\omega^{q} (\Omega)}  + \norm{d}_{L_\omega^{q} (\Omega)}  +  \norm{ p}_{(\hat W^{1,q'}_{0, \omega'} (\Omega))^*}   \right).
 \end{aligned}
 \end{equation}
 To conclude, we need to show that \eqref{eq:a-prior1} implies the thesis \eqref{aoprioriconst}. To this end we will use the classical Agmon-Douglis-Nirenberg reasoning by contradiction.  Recall that we work, by assumption, with $(w,p) \in W_\omega^{1,q}(\Omega) \times L_\omega^{q}(\Omega)$. Assume that \eqref{aoprioriconst} is false, \emph{i.e.} that there is a sequence $(w_j,p_j) \in W_\omega^{1,q}(\Omega) \times L_\omega^{q}(\Omega)$, $F_j, d_j \in L_\omega^{q}(\Omega)$ solving  \eqref{eq:linSt} such that
\[ C_j := \norm{ w_j}_{W_\omega^{1,q} (\Omega)} +  \norm{  p_j - \langle p_j \rangle}_{L_\omega^{q} (\Omega)} \ge j (\norm{F_j}_{L_\omega^{q} (\Omega)} + \norm{d_j}_{L_\omega^{q} (\Omega)} )
\]
 Due to linearity of   \eqref{eq:linSt}, $W_j := \frac{w_j}{C_j}$ and $P_j := \frac{p_j - \langle p_j \rangle }{C_j}$ solve \eqref{eq:linSt} with force $R_j := \frac{F_j}{C_j}$ and compressibility $D_j := \frac{d_j}{C_j}$. Observe we have  $\langle P_j \rangle=0$.  Hence we have by our  above assumption 
 \[ 1=  \norm{W_j}_{W_\omega^{1,q} (\Omega)} +  \norm{P_j}_{L_\omega^{q} (\Omega)} \ge j \left( \norm{R_j}_{L_\omega^{q} (\Omega)} +  \norm{D_j}_{L_\omega^{q} (\Omega)}  \right). 
\] 
It means that we can find a (non-relabeled) sub-sequence and respective limits:
\[
\begin{aligned}
 % \label{conn-a}
  \nabla W_j &\to \nabla W_\infty &&\textrm{weakly in }L^{q}_{\omega}(\Omega) , \quad \textrm{strongly in } {(W^{1,q'}_{0, \omega'} (\Omega))^*}  ,
  \\
	 W_j &\to W_\infty &&\textrm{weakly in } W^{1,q}_{\omega}(\Omega), \quad \textrm{strongly in } L^{q}_{\omega}(\Omega),
  \\
% \label{conn-b}
P_j  &\to P_\infty &&\textrm{weakly in }
 L^{q}_{\omega}(\Omega),  \quad \textrm{strongly in } (W^{1,q'}_{0, \omega'}(\Omega))^*,
  \\
R_j  &\to 0, D_j  \to 0 &&\textrm{strongly in }
 L^{q}_{\omega}(\Omega),
\end{aligned}
\]
where first two strong limits follow from compact embeddings $W^{1,q}_{\omega}  \hookrightarrow L^{q}_{\omega} \hookrightarrow (W^{1,q'}_{0, \omega'})^*$, see Theorem 2.3 of \cite{Fro07} for the former and Lemma~\ref{lem:dualcomp} for the latter.

 Moreover, taking limit $j \to \infty$ in  \eqref{eq:linSt} solved by $(W_j, P_j)$, with data $R_j, D_j, 0$ we see that $(W_\infty, P_\infty) = (0,0)$ in view of uniqueness of the zero solution (with zero mean pressure) to  \eqref{eq:linSt}. {The uniqueness of the zero solution follows, for instance, from the fact that within Muckenhoupt weights we have for a bounded $\Omega$ that $L^p_\omega (\Omega)
\embedding L^s (\Omega)$ for a certain $s>1$. Consequently, we can use a classical uniqueness theorem in $ L^s$, which can be found for instance in Section 3 of   Borchers and Miyakawa \cite{BorMiy90}.}
Hence
\begin{multline*}
1= \norm{W_j}_{W_\omega^{1,q} (\Omega)} +  \norm{P_j}_{L_\omega^{q} (\Omega)} \le \\
C (q, A_q, \Omega) \left( \norm{R_j}_{L_\omega^{q} (\Omega)} +  \norm{ W_j}_{L_\omega^{q} (\Omega)}  +  \norm{D_j}_{L_\omega^{q} (\Omega)} +  \norm{P_j}_{(W_{0, \omega,}^{1,q'} (\Omega))^*}  \right) \stackrel{j \to 0}{\to 0},
\end{multline*}  
which contradicts \eqref{eq:a-prior1}.
 
We have reached the thesis \eqref{aoprioriconst} for $g=0$. In order to include the non-homogenous case  $g \neq 0$, {recall that the trace space $ T^{q}_\omega(\Omega)$ (or its homogenous version $ \hat T^{q}_\omega(\Omega)$)  is defined via the existence of an extension} $\gamma^{-1}:   T^{q}_\omega(\Omega) \to W_{\omega}^{1,q} (\Omega)$, which is linear and bounded). {Therefore \eqref{eq:linSt} can be transferred into
$(\tilde{w},p):=(w-\gamma^{-1}g,p)$, which is a solution to the following system:}
  \begin{equation*}
  \begin{aligned}
    -\divergence (\epsilon \tilde{w}) + \nabla p &=-\divergence (F-\epsilon {\gamma^{-1}} g) \quad \text{ in } \Omega,\\
    \divergence \tilde{w}&=d-\divergence(  {\gamma^{-1}} g) \qquad \quad \; \text{ in } \Omega,\\
    \gamma(\tilde{w})&=0 \qquad \quad \;  \text{ on } \partial\Omega.
  \end{aligned}
  \end{equation*}
 and the result can be achieved using the estimate for homogeneous boundary data.
%\comentario{J: this case $g \neq 0$ is nicely short now. I have only one little doubt: is it obvious that $\gamma$ for weighted spaces, defined merely as a domain restriction of a trace operator for $W^{1,1}$, is bounded in weighted spaces? }
\end{proof}

\subsection{Proof of Theorem~\ref{thm:a-priori}} Recall for Section \ref{sec:muck} that due to the {\em miracle of extrapolation}, it is sufficient to proof the desired estimates in the case $L^2_{\omega}(\Omega)$, with  $\omega\in A_2$. By our assumption, $(v, \pi)$  solves \eqref{eq:sysA} and $\nabla v, \pi \in L^s(\Omega)$ for some $s\in (1,\infty)$. Due to boundedness of $\Omega$, we can assume without loss of generality that $s\in (1,2]$.
The first idea behind our estimate is to approximate $\omega$ by $\omega_j$ such that $\nabla v, \pi \in L^2_{\omega_j}(\Omega)$.  By \eqref{weight1}, we have for $\tilde \omega_1=(M {\nabla v})^{s-2}\in A_2$ and $\nabla u \in L^2_{\omega_1}(\Omega)$ as well as for $\tilde \omega_2=(M {p})^{s-2}\in A_2$ and $p \in L^2_{\omega_2}(\Omega)$. Let us take $\tilde \omega_3=\min\set{\tilde \omega_1, \tilde \omega_2}$ and $\omega_j=\min\set{j \tilde \omega_3,\omega}$. Obviously, $\nabla u\in  L^2_{\omega_j}(\Omega)$ and $f\in L^2_{\omega_j}(\Omega)$. But moreover, by \eqref{eq:A2min}, we find that $A_2(\omega_j)\leq  A_2(\omega)+ A_2(\omega_3)$, since $A_q(\omega_1)=A_q(j\omega_3)$ directly by definition.
For this $\omega_j$ we perform now  the following \emph{a-priori} estimate.

%\comentario{It is sufficient to consider the case of \eqref{eq:sysA} with $\mu=1$, thanks to linearity.} 

Let us rewrite  \eqref{eq:sysA} as a distributional formulation of the linear Stokes problem
\begin{equation}
\int_{{\Omega}} \mu \;  \D \cdot \nabla \phi +\pi\divergence \phi\dx = \int_{{\Omega}}(f -\Scal(x,{\D})  + \mu \; \D)\cdot \nabla \phi. \label{wfn2}
\end{equation}
Since $\nabla v\in  L^2_{\omega_j}$, we can use estimate of Lemma~\ref{lem:CZ} and Assumption \ref{ass:A} to provide the following absorption with $C = C(A_2(\omega)+ A_2(\omega_3), \Omega)$
\begin{align*}
& \norm{\nabla v}^2_{L^2_{\omega_j} (\Omega)}+ \norm{\pi - \langle \pi \rangle}_{L^2_{\omega_j}(\Omega)}^2\leq C \int_{\Omega}\abs{f}^2 {\omega_j} + \abs{\Scal (x,{\D})- \mu \; {\D}}^2{\omega_j}\\
&\leq C\int_{\Omega}(\abs{f}^2+ 2 c^2_1 m^2 + 2 c^2_2 + 2 \mu^2 m^2) \, {\omega_j} + C \int_{\{|{\D}|\ge m \}}\frac{\abs{\Scal (x,{\D})- \mu \; {\D}}^2}{|{\D}|^2} |{\D}|^2{\omega_j}.
\end{align*}
Due to the assumed linearity-at-infinity we can find such $m = m_0$ that the last summand on the r.h.s. above does not exceed half of the first on of the l.h.s. Consequently
\begin{align}\label{pre_finaln}
\norm{\nabla v}_{L^2_{\omega_j} (\Omega)}  +  \norm{\pi - \langle \pi \rangle}_{L^2_{\omega_j}(\Omega)}\leq  C(A_2(\omega)+ A_2(\omega_3), \mu, \Omega)\left( 1 + \norm{f}_{L^2_{\omega_j} (\Omega)} \right).
\end{align}
Observe that the above constant is $j$-uniform.
Next, we let $j\to \infty$ in \eqref{pre_finaln}. For the right hand side, we use the fact that $\omega_j \leq \omega$ and for the left hand side we use the
monotone convergence theorem (notice here that $\omega_j\nearrow \omega$ since $\omega_3 < \infty$ almost everywhere). Consequently
\begin{align}\label{final}
\norm{\nabla v}_{L^2_{\omega} (\Omega)} + \norm{\pi -  \langle \pi \rangle}_{L^2_{\omega}(\Omega)} \leq C(A_2(\omega)+ A_2(\omega_3), \Omega)  \left(1+  \norm{f}_{L^2_{\omega} (\Omega)} \right).
\end{align}
This implies the quantitative estimate, but with $C$ still depending on $A_2(\omega_3)$. Therefore we use from \eqref{final} only the qualitative information $\nabla v, \pi \in  L^2_{\omega}$ and redo the absorption for $\omega$ alone. Consequently one gets the desired estimate with dependence on $A_2(\omega)$ alone. Therefore the extrapolation~\cite[Theorem~1.4]{CruMarPerBook} can be applied and the theorem is proved.
\qed

\section{Proofs of the technical results}\label{sec:techex}
This section contains proofs of  Theorem~\ref{T5} (solenoidal div-curl lemma) and   Theorem \ref{thm:liptrunc} (solenoidal Lipschitz truncations). Let us begin with the latter, since it is needed in the proof of the former.

\subsection{Lipschitz truncations}
%\comentario{J: I did not touch this part, this is your 'homework' copy-pasted. I think it would be nicer to have it already here stated as a lemma for $u$ on domain via approximations of identity that you somehow mention directly after the proof. Then we can use it in the main proof without any extra fuss concerning balls, that would be unclear there}
Since even the optimal regularity of \eqref{eq:sysA} for $q <2$ is insufficient for $u$ to be a test function, we resort to Lipschitz truncations. {It is  a standard tool by now, originally developed in Acerbi \& Fusco \cite{AceF88}, Frehse, M\'alek \& Steinhauer \cite{FreMalSte00}, see also Diening, M\'alek \& Steinhauer \cite{DieMalSte08}. Recently a further advance was provided, that is important for the fluid dynamics considerations, namely a solenoidal Lipschitz truncation, see Breit, Diening \& Fuchs ~\cite{BreDieFuc12}, and  Breit, Diening \& Schwarzacher \cite{BreDieSch13}. Let us present weighted estimates for the solenoidal Lipschitz truncations developed in~\cite{BreDieSch13}} and fine-tune them for our purposes.
\begin{lemma}[Solenoidal Lipschitz approximation on balls]
  \label{thm:liptrunc1}
  Let $B \subset \setR^n$ be a ball and $s>1$. Let $g \in W^{1,s}_{0,\divergence}(B)$. Then, for
  all $\lambda >\lambda_0$, there exists a \emph{Lipschitz truncation} $g^\lambda \in
  W^{1,\infty}_{0,\divergence}(2B)$ such that
  \begin{alignat}{2}
    \label{eq:lip1l}
    g^\lambda &= g \quad \text{and} \quad \nabla g^\lambda = \nabla g
    &\qquad&\text{in $\set{M(\nabla g)\leq \lambda}\subset 2B$},
    \\
    \label{eq:lip2l}
    \abs{\nabla g^\lambda} &\leq \abs{\nabla g}\chi_{\set{M(\nabla
        g)\leq\lambda}}+C\, \lambda
    \chi_{\set{M(\nabla g)>\lambda}} &&\textrm{almost
      everywhere}.
  \end{alignat}
  Further, if $\nabla g\in L^{p}_\omega(\Omega;\setR^{n\times N})$ for
  some $1\leq p<\infty$ and $\omega \in {A}_p$, then
  \begin{align}\label{itm:weight}
    \begin{aligned}
      \int_{2B}\abs{\nabla g^\lambda}^p \omega \dx &\leq
      C\int_{B} \abs{\nabla g}^p \omega \dx,
      \\
      \int_{2B} \abs{\nabla (g-g^\lambda)}^p \omega \dx&\leq
      C \int_{B\cap \set{M(\nabla g)>\lambda}} \abs{\nabla g}^p \omega \dx,
    \end{aligned}
  \end{align}
where the constant $C$ depends on $({A}_p(\Omega),\Omega, N, p)$ and $\lambda_0=c(s,n)\bigg(\dashint_B\abs{\nabla g}^s\dx\bigg)^\frac1s$.
\end{lemma}
\begin{proof}
All statements except for~\eqref{itm:weight}
are already contained in \cite[Lemma 4.3 \& Theorem~4.4]{BreDieSch13}. Please observe, that although the construction there is done in the three dimensional case, the arguments are in fact valid in all dimension by replacing the inverse-$\curl$ operator with its $n$-dimensional analogue as defined in Remark~2.18 in \cite{BreDieSch13}.

 The
first inequality of~\eqref{itm:weight} follows directly from the
second, so it is enough to prove the latter. 

Let us extend both $g$ and $g^\lambda$ by $0$ outside $B$ and $2B$ respectively. It follows from~\eqref{eq:lip1l} and \eqref{eq:lip2l} that
\begin{equation}
  \label{eq:lipA1}
\begin{aligned}
  \norm{\nabla (g-g^\lambda)}_{L^p_\omega(\setR^n)}&= \norm{\nabla
    (g-g^\lambda) \chi_{\set{M(\nabla g)>\lambda}}}_{L^p_\omega(\setR^n)}
  \\
  &\leq \norm{\nabla g\, \chi_{\set{M(\nabla g)>\lambda}}}_{L^p_\omega(B)}
  + C\, \norm{\lambda\, \chi_{\set{M(\nabla g)>\lambda}}}_{L^p_\omega(\setR^n)}.
\end{aligned}
\end{equation}
The second term will be estimated by a \Calderon-Zygmund-type covering argument. As $\set{M(\nabla g)>\lambda}\subset 2B$ is open, for every $x
\in \set{M(\nabla g)>\lambda}$ there is a ball $B_{r(x)}(x)\subset \set{M(\nabla g)>\lambda}$ such that
\begin{align}
  \label{eq:lip5}
  \lambda< \dashint_{B_r(x)}\abs{\nabla g} dx\leq 2\lambda.
\end{align}
%\comentario{J: we don't use $\leq 2\lambda$ in what follows, do we?}
These balls cover $\set{M(\nabla g)>\lambda}$.  Next, using the Besicovich
covering theorem, we extract from this cover a countable subset
$B_i$ which is locally finite, i.e.
\begin{align}
  \label{eq:lip6}
  \# \{j\in \mathbb{N}; \, B_i\cap B_j \neq \emptyset\} \le C(n).
\end{align}
In the following, for a measurable set $A$ we write $| A |_\omega = \int_A \omega dx$. Using~\eqref{eq:lip5},   \eqref{defAp2} and \eqref{eq:lip6}, we have the following estimate
\begin{align*}
&  \norm{\lambda\, \chi_{\set{M(\nabla g)>\lambda}}}_{L^p_\omega(\setR^n)}^p =
  \lambda^p |\set{M(\nabla g)>\lambda}|_\omega \leq \sum_i \lambda^p
 | B_i|_\omega \leq \sum_i \bigg( \dashint_{B_i} \abs{\nabla g}\,dx \bigg)^p
|B_i|_\omega \\
 & \leq \sum_i \dashint_{B_i} \abs{\nabla g}^p \omega\,dx
  \bigg(\dashint_{B_i} \omega^{-(p'-1)}\,dx \bigg)^{\frac{1}{p'-1}}
|B_i|_\omega \leq {A}_p(\omega) \sum_i \int_{B_i} \abs{\nabla g}^p
  \omega \,dx \\
&	\leq C(n) \,{A}_p(\omega) \int_{\set{M(\nabla g)> \lambda}}
  \abs{\nabla g}^p \omega \,dx = C(n) \,{A}_p(\omega) \int_{B}
  \abs{\nabla g}^p  \chi_{\set{M(\nabla g)>\lambda}} \omega \,dx.
\end{align*}
This directly leads to the following inequality
\begin{align*}
  \norm{\lambda\, \chi_{\set{M(\nabla g)>\lambda}}}_{L^p_\omega(\setR^n)} &\leq
  C(n) \,{A}_p(\omega)^{\frac 1p} \norm{\nabla g\,
    \chi_{\set{M(\nabla g)>\lambda}}}_{L^p_\omega(B)},
\end{align*}
which used in \eqref{eq:lipA1}  finishes the proof of the desired estimate~\eqref{itm:weight}.
\end{proof}
Next, we provide proof of Theorem \ref{thm:liptrunc}. We loose the zero trace of its counterpart on balls from the preceding  Lemma \ref{thm:liptrunc1}, but deal with Lipschitz truncation on general domains $\Omega$.

\begin{proof}[Proof of Theorem \ref{thm:liptrunc}]
We use the construction of~\cite[Section~4]{BreDieSch13}. The fact that $g$ has zero trace in a ball is used only in Lemma 4.2 and Lemma 4.3 there, so all other results can be directly applied to our situation. The construction of $g^\lambda$ and Lemma~4.1 of \cite{BreDieSch13} are valid in all dimensions and for a general domain $\Omega$ with no changes, except for the replacing of the inverse $\curl$-operator with the $n$-dimensional analogue, as  defined in Remark~2.18 of \cite{BreDieSch13}. Moreover, by using for general $\Omega$ the local estimates intended for balls in \cite{BreDieSch13}, one looses only information of the zero trace, but all estimates hold and the solenoidality is preserved. For instance, the argument in the proof of Lemma 4.3 implies in our case that $g^\lambda\in W^{1,1}_{\divergence}(\Omega)$ (no zero trace). Moreover, the proof of Theorem~4.4 of \cite{BreDieSch13} implies all the needed by us above assertions, with the exception of \eqref{itm:weight2}. These weighted estimate follow from redoing the argument in Lemma~\ref{thm:liptrunc1}, by replacing there $B$ with $\Omega$.
\end{proof}

\subsection{Solenoidal, generalized div-curl lemma}
Let us focus now on the proof of Theorem \ref{T5}. 
It is divided into several steps for clarity.
\subsubsection{Preliminary Step 0.}
{ Firstly, by the reflexivity and separability of $L^{q}_\omega, L^{q'}_\omega$ together with the assumption  \eqref{bit3} and due to the embedding $L^{q'}_\omega (\Omega)
\embedding L^{1+ \delta}(\Omega)$, compare with \eqref{eq:lqprop}, we find a subsequence} 
  \begin{equation}\label{bit:p:2}
  s^k \rightharpoonup s \quad \textrm{weakly in } L^{q'}_\omega(\Omega)\cap L^1(\Omega),  \qquad   a^k \rightharpoonup a \quad \textrm{weakly in } L^q_\omega(\Omega)\cap L^1(\Omega).
  \end{equation}
 { In the following we show the remaining  \eqref{bitf}. Since we aim to show convergence on a (large) subset of $\Omega$ we may assume without loss of generality that $\partial\Omega$ is $C^\infty$-smooth.}
\subsubsection{Step 1. Reduction to the non-solenoidal case}
Let us consider the linear Stokes problem
  \begin{equation}\label{eq:linStDC}
  \begin{aligned}
    -\divergence (\epsilon ({w^k})) + \nabla p^k &=-\divergence s^k \quad \text{ in } \Omega,\\
    \divergence w&=0 \qquad \quad \; \text{ in } \Omega
  \end{aligned}
  \end{equation}
with null boundary-values. Lemma \ref{lem:CZ} and assumption \eqref{bit3} imply that
 \begin{equation}\label{bit:p:1}
\norm{\nabla w^k}_{L^{q'}_\omega (\Omega)} +\norm{p^k}_{\hL^{q'}_\omega(\Omega)}  \le C \left(1+ \norm{s^k}_{L^{q'}_\omega (\Omega)}\right) \le C.
\end{equation}
And hence assumption  \eqref{bit3} and the embedding \eqref{eq:lqprop} implies that we may pass to a subsequence, such that
    \begin{align}\label{bit:p:4}
p^k \rightharpoonup p &&\textrm{weakly in } L^{q'}_\omega (\Omega).
  \end{align}
  
Let us consider $b^k =: s^k + p^k \Id$. Assume for a moment that for every bounded sequence
  $\{c^k\}_{k=1}^{\infty}$ in $W^{1,\infty}_{0} (\Omega)$ such that
  $$
  \nabla c^k \rightharpoonup^* 0 \qquad \textrm{weakly$^*$ in }
  L^{\infty}(\Omega)
  $$
one has
  \begin{equation}\label{bit:p:3}
    \lim_{k\to \infty} \int_{\Omega} b^k \cdot \nabla c^k \dx
    =0.
    \end{equation}
Then, making in the non-solenoidal, weighted, biting div-curl lemma, \emph{i.e.} Theorem 2.6 of \cite{BulDieSch15} the following choices
\[
a^k =: a^k, \quad b^k =: b^k,
\]
we see  via our assumptions  and \eqref{bit:p:1} that  the assumptions of the non-solenoidal lemma are satisfied. It thesis implies existence of a subsequence such that 
  \begin{align}
  a^k &\rightharpoonup a &&\textrm{weakly in } L^1(\Omega), \label{bit:p:6a} \\
  b^k &\rightharpoonup b &&\textrm{weakly in } L^1(\Omega), \\
  a^k \cdot b^k \omega &\rightharpoonup a \cdot b\, \omega &&\textrm{weakly in } L^1(\Omega_j) \quad \textrm{ for all } j\in \mathbb{N}. \label{bit:p:6}
  \end{align}
  
  Due to \eqref{bit:p:2}, we identify $b =  s + p \Id$.
  % {\seb{Next, observe that as $p^k \Id \cdot  a^k \omega$ is uniformly bounded in $L^1(\Omega)$, we find (by maybe decrasing $\Omega_j$ slightly) by Theorem~\ref{thm:blem}} that cthere is a $c\in L^1(\Omega)$, such that
%\[
%p^k \Id \cdot  a^k \omega\weakto c\in L^1(\Omega_j)\text{ for all }j\in \setN.
%\]  
Finally, assumption \eqref{bit6} gives, after decreasing $\Omega_j$ slightly, via Egoroff's theorem
  \begin{equation}\label{bit:p:7}
  p^k \Id \cdot  a^k \omega =   p^k \tr (a^k) \omega \rightharpoonup p \tr (a) \omega =   p \Id \cdot a \, \omega \qquad  \textrm{weakly in } L^1(\Omega_j),
  \end{equation}
  thanks to \eqref{bit:p:4}, uniqueness of the limiting $a$  and the strong-weak coupling.%} 
  
  Subtracting from \eqref{bit:p:6} with $b^k =: s^k + p^k \Id$ and $b =  s + p \Id$ the formula \eqref{bit:p:7} we arrive at \eqref{bitf}. The limits \eqref{bitfa}, \eqref{bitfb} are given as \eqref{bit:p:2} and \eqref{bit:p:6a}. 
  
 Consequently, we are left with justifying the compactness condition \eqref{bit:p:3}. Since
 the first equation of \eqref{eq:linStDC} can be rewritten as 
 \begin{equation}\label{bit:p:10}
  \divergence b^k =\divergence \nabla w^k,
 \end{equation}
the condition \eqref{bit:p:3} is equivalent to the strong-$L^1$ precompactness of $\nabla w^k$. We will accomplish this in the following three steps.
 
 \subsubsection{Step 2. Solenoidal truncations}

Let us use Theorem  \ref{thm:liptrunc} to truncate solenoidally  $w^k$ at height $\lambda$, producing $w^{k, \lambda}$.  For the following {\em dual forcing} given by
  \[Q (\eta) := |\eta|^{q'-2} \eta,\]
 %  \comentario{J: I have changed to $q'$, compared to MB suggestion, for the sake of \eqref{eq:wL12:2}. Check off it does not spoil something}
  let us consider the following auxiliary linear Stokes problem
  \begin{equation}\label{eq:linStDC2}
  \begin{aligned}
    -\divergence (\epsilon ({z^{k,\lambda }})) + \nabla t^{k, \lambda} &=-\divergence Q(\nabla w^{k, \lambda})  &&\quad \text{ in } \Omega,\\
    \divergence z^{k, \lambda}&=0  && \quad  \text{ in } \Omega
  \end{aligned}
  \end{equation}
with null boundary-values. 
Boundedness of $Q(\nabla w^{k, \lambda})$ for a fixed $\lambda$ and Lemma \ref{lem:CZ} imply that for any finite $p$ one has
\[
\norm{z^{k,\lambda }}_{W^{1, p}_{0, div} }+ \norm{t^{k,\lambda }}_{L^{ p} }  \le C (\lambda)
\]
 and the regularity is inherited by the limiting equation with respect to $k \to \infty$, that reads
   \begin{equation}\label{eq:linStDC2:l}
  \begin{aligned}
    -\divergence (\epsilon ({z^\lambda})) + \nabla t^\lambda &=-\divergence \, Q_\lambda  &&\quad \text{ in } \Omega,\\
    \divergence z^\lambda&=0 &&\quad \text{ in } \Omega
  \end{aligned}
  \end{equation}
  with null boundary-values. The above $Q_\lambda$ denotes the $L_\omega^{q}$ weak limit of $Q(\nabla w^{k, \lambda}) $ (since $Q(\nabla w^{k})$, hence $Q(\nabla w^{k, \lambda}) $  is $k$-uniformly bounded in $L_\omega^{q}$). 
  
 For a non-relabeled subsequence of  $Q_\lambda$, let us immediately denote its $L_\omega^{q}$ weak limit by $Q_0$.

 \subsubsection{Step 3. A non-weighted weak-$L^1$ limit for truncations}

Our  aim in this step is to show, for a fixed $\lambda$ (possibly, again on a non-relabeled subsequence) that for $k \to \infty$
\begin{equation}\label{Minty2:1}
 Q(\nabla w^{k, \lambda})    \cdot \nabla w^{k}  \rightharpoonup Q_\lambda
  \cdot  \nabla w \qquad \textrm{weakly in } L^1(\Omega).
\end{equation}
 Due to  \eqref{bit:p:1} and boundedness of $\nabla w^{k,\lambda}$, we see that $ Q(\nabla w^{k, \lambda})    \cdot \nabla w^{k} $ is $k$-uniformly $L_\omega^{q'} \subset L^{1+ \delta} $ integrable, hence equiintegrable. Consequently, it possesses a weakly-$L^1$ converging subsequence. Now, to identify it with $Q_\lambda
  \cdot  \nabla w$, it suffices to show that for all $\eta \in \mathcal{D}(\Omega)$ we have
\begin{equation}\label{Minty2:2:p}
\lim_{k\to \infty} \int_{\Omega}  Q(\nabla w^{k, \lambda})    \cdot \nabla w^{k}  \;  \eta = \int_{\Omega}  Q_\lambda
  \cdot  \nabla w \; \eta.
\end{equation}
Let us write
\[
\int_{\Omega}  Q(\nabla w^{k, \lambda})    \cdot \nabla w^{k}  \eta  = \int_{\Omega}  (Q(\nabla w^{k, \lambda}) -  \epsilon (z^{k, \lambda})  )   \cdot \nabla w^{k}  \eta +  \int_{\Omega}   \epsilon (z^{k, \lambda}) \cdot \nabla w^{k}  \eta =: I^{k,\lambda} + II^{k,\lambda} 
\]
One has
\[
\begin{aligned}
 I^{k,\lambda}  = &\int_{\Omega} \left(Q(\nabla w^{k, \lambda}) - \epsilon (z^{k, \lambda})  \right) \cdot \nabla (w^{k} \eta) \dx  - \int_{\Omega}  \left(Q(\nabla w^{k, \lambda}) - \epsilon (z^{k, \lambda})  \right) \cdot   \left( w^{k} \otimes \nabla \eta \right)  \dx \\
=& \int_{\Omega} t^{k, \lambda} \divv (w^{k} \eta) \dx  - \int_{\Omega}  \left(Q(\nabla w^{k, \lambda}) - \epsilon (z^{k, \lambda})  \right) \cdot   \left( w^{k} \otimes \nabla \eta \right)  \dx \\
=& \int_{\Omega} t^{k, \lambda}  w^{k} \nabla \eta \dx  - \int_{\Omega}  \left(Q(\nabla w^{k, \lambda}) - \epsilon (z^{k, \lambda})  \right) \cdot   \left( w^{k} \otimes \nabla \eta \right)  \dx
\end{aligned}
\]
where for the second equality above we used the equation \eqref{eq:linStDC2} and for the last one - solenoidality of $w^{k}$. We have obtained formulas with  a coupling of $w^{k}$, strong-converging  in $L_\omega^{q'} \subset L^{1+ \delta}$   and the remainders weak converging in $L^p$ with any finite $p$. Hence we can pass to the limit and recover it by reverse equalities as follows
\[
\begin{aligned}
\lim_{k \to \infty}  I^{k,\lambda}   = & \int_{\Omega} t^{ \lambda}  w \nabla \eta \dx  - \int_{\Omega}  \left(Q_\lambda - \epsilon (z)  \right) \cdot   \left( w \otimes \nabla \eta \right)  \dx \\
= & \int_{\Omega} t^{ \lambda}  \divv (w^{k} \eta)  \dx  - \int_{\Omega}  \left(Q_\lambda - \epsilon (z)  \right) \cdot   \left( w \otimes \nabla \eta \right)  \dx \\
=& \int_{\Omega}  (Q_\lambda -  \epsilon (z^{ \lambda})  )   \cdot \nabla w \;  \eta.
\end{aligned}
\]

Function $w^{k} \eta$ with the Bogovskii correction\footnote{Compare Bogovskii~\cite{Bog79,Bog80} and Diening, R{\r u}{\v z}i{\v c}ka \& Schumacher  \cite{DieRS10}}  is admissible in \eqref{eq:linStDC2}. Therefore we can write  for $II^{k,\lambda}$
\[
\begin{aligned}
II^{k,\lambda} =&  \int_{\Omega}  \epsilon ( w^{k}) \cdot  \nabla (z^{k, \lambda} \eta) dx - \int_{\Omega}  \epsilon ( w^k) \cdot (z^{k, \lambda}\otimes \nabla \eta) dx \\
=& \int_{\Omega} \epsilon ( w^{k}) \cdot \nabla (z^{k, \lambda}  \eta - \Bog (z^{k, \lambda} \otimes \nabla \eta))  dx+ \int_{\Omega} \epsilon ( w^{k})  \cdot \nabla \left( \Bog (z^{k, \lambda} \otimes \nabla \eta )\right)  dx \\
&  - \int_{\Omega}   \nabla w^k \cdot (z^{k, \lambda} \otimes \nabla \eta) dx  \\
=& \int_{\Omega} s^k  \cdot \nabla (z^{k, \lambda}  \eta - \Bog (z^{k, \lambda} \otimes \nabla \eta))  dx+ \int_{\Omega} \nabla w^k  \cdot \nabla \left( \Bog (z^{k, \lambda} \otimes \nabla \eta )\right)  dx \\
&  - \int_{\Omega}   \nabla w^k \cdot (z^{k, \lambda} \otimes \nabla \eta) dx,
\end{aligned}
\]
where for the second equality above we used, this time, the equation \eqref{eq:linStDC}.
We use our assumption \eqref{bit5} to pass to the limit in the first term above. For the last two terms, we invoke continuity of Bogovskii operator in $L^p$ spaces to pass to the respective limits, thanks to the strong-weak coupling. Using for the limit \eqref{eq:linStDC} to reverse, we see that
\[
\lim_{k \to \infty}  II^{k,\lambda}   =  \int_{\Omega}   \epsilon (z^{ \lambda})   \cdot \nabla w \;  \eta.
\]
Putting together limits for $ I^{k,\lambda}$  and $ II^{k,\lambda}$, we obtain \eqref{Minty2:2:p}, thus \eqref{Minty2:1}. 

 \subsubsection{Step 4. A weighted weak-$L^1$ biting limit}
 Our goal here is to show that $Q (\nabla w^k) \cdot\nabla {w^k} \omega$ tends to $Q_0 \cdot  \nabla w  \; \omega$ weakly in $L^1(\Omega)$; in fact, we will have to decrease $\Omega$ slightly. Recall that \eqref{Minty2:1} does not involve a weight $\omega$. Therefore, we decompose an arbitrary $\omega \in A_{q'}$ as follows
\[
\omega = \frac{\omega}{1 + \delta \omega} +  \frac{\delta \omega^2}{1 + \delta \omega},
\]
with the former summand bounded for any $\delta >0$. 
Let us write
\begin{equation}\label{eq:s43:1:p}
 \begin{aligned}
& Q (\nabla w^k) \cdot\nabla {w^k} \omega - Q_\lambda
  \cdot  \nabla w  \; \omega \\
& =   \left( Q(\nabla w^{k, \lambda})    \cdot \nabla w^{k}  - Q_\lambda
  \cdot  \nabla w   \right) \omega  +  (Q(\nabla w^{k})   - Q(\nabla w^{k, \lambda}) )    \cdot \nabla w^{k} \; \omega \\
  & = \left( Q(\nabla w^{k, \lambda})    \cdot \nabla w^{k}  - Q_\lambda
  \cdot  \nabla w \right)  \frac{ \omega }{1 + \delta \omega} +  \left( Q(\nabla w^{k, \lambda})    \cdot \nabla w^{k}  - Q_\lambda
  \cdot  \nabla w \right)  \frac{\delta \omega^2 \id_{\{\omega \le \lambda\}} }{1 + \delta \omega}   \\
&   + \left( Q(\nabla w^{k, \lambda})    \cdot \nabla w^{k}  - Q_\lambda
  \cdot  \nabla w \right)  \frac{\delta  \omega^2  \id_{\{\omega > \lambda\}}  }{1 + \delta \omega}  +  (Q(\nabla w^{k})   - Q(\nabla w^{k, \lambda}) )    \cdot \nabla w^{k} \; \omega \\
  & =:  III_\delta^{k,\lambda}  +IV_\delta^{k,\lambda} +V_\delta^{k,\lambda} +VI^{k,\lambda}.
 \end{aligned}
 \end{equation}
 We will deal with $III_\delta^{k,\lambda}$ and $IV_\delta^{k,\lambda}$ directly via \eqref{Minty2:1}. Indeed, \eqref{Minty2:1} extends automatically to its weighted version, as long as the involved weight is  bounded. Therefore, as for fixed $\lambda, \delta$ the respective weights are bounded,  we have for an arbitrary $\psi \in L^\infty (\Omega)$
\begin{equation}\label{eq:s43:1:3}
 \lim_{k \to \infty} \int
  III_\delta^{k,\lambda}  \psi = 0, \qquad \lim_{k \to \infty} \int
  IV_\delta^{k,\lambda}  \psi = 0
 \end{equation}
 
In relation to   $V_\delta^{k,\lambda}$ we write, using the H\"older inequality
\begin{equation}\label{eq:s43:1:4}
\begin{aligned}
& \int V_\delta^{k,\lambda} \psi \le  \norm{\psi}_\infty  \int  \left( |Q(\nabla w^{k, \lambda}) | |\nabla w^{k}| + |Q_\lambda| |\nabla w| \right)   \frac{\delta \omega^2 \id_{\{\omega > \lambda\}} }{1 + \delta \omega}  \\
&  \le \norm{\psi}_\infty \norm{Q(\nabla w^{k, \lambda}) }_{L^q_\frac{\delta \omega^2  \id_{\{\omega > \lambda\}} }{1 + \delta \omega} } \norm{\nabla w^{k} }_{L^{q'}_\frac{\delta \omega^2 \id_{\{\omega > \lambda\}}  }{1 + \delta \omega} } + \norm{ \psi}_\infty \int  |Q_\lambda| |\nabla w|    \frac{\delta \omega^2 }{1 + \delta \omega} \\
&  \le \norm{\psi}_\infty  \norm{\nabla w^{k} }^2_{L^{q'}_{\omega \id_{\{\omega > \lambda\}} } } + \norm{ \psi}_\infty  \int  |Q_\lambda| |\nabla w|    \frac{\delta \omega^2}{1 + \delta \omega},
  \end{aligned}
 \end{equation}
 where for the second inequality we used growth of $Q$, \eqref{itm:weight2} and $  \frac{\delta \omega^2}{1 + \delta \omega} \le \omega$ almost everywhere.

Let us imply the Biting Lemma~\ref{thm:blem} on the sequence $\abs{\nabla {w^k}}^{q'}\omega$, compare \eqref{bit:p:1}. Consequently, there is a sequence $\Omega_j$  such that $|\Omega \setminus \Omega_j|\to 0$  and for any $K\subset \Omega_j$ holds
 \begin{equation}\label{smallunif}
\int_K \abs{\nabla {w^k}}^{q'}\omega \leq \epsilon.
\end{equation}
$k$-uniformly, as long as $\abs{K}\leq \delta_{\varepsilon, j}$.  The Chebyshev inequality for $\omega$, integrable by definition, indicates, that the role of $K$ may play $\set{\omega >\lambda}$ for sufficiently large $\lambda$, as long as we restrict ourselves to $\Omega_j$ in \eqref{eq:s43:1:3} and \eqref{eq:s43:1:4}. 
Indeed, in tandem with the above application of the biting Lemma, for every $j$ and $\epsilon$ there exists $\lambda_j^\epsilon$, such that
 \begin{align}
 \label{eq:wL12:p}
\int_{\set{ w>\lambda}\cap \Omega_j}\abs{\nabla {v^k}}^{q'}\omega \leq \epsilon\quad \text{ for every }\quad \lambda\geq \lambda_j^\epsilon.
 \end{align}
Consequently, a restriction to $\Omega_j$ does not change  \eqref{eq:s43:1:3} and allows to write via use of \eqref{eq:wL12:p} in \eqref{eq:s43:1:4} that for any $\epsilon$ and each $ \lambda\geq \lambda_j^\epsilon$
\begin{equation}\label{eq:s43:1:4'}
 \int_{\Omega_j} V_\delta^{k,\lambda} \psi  \le C\epsilon+ C (\norm{ \psi}_\infty)   \int_{\Omega_j}   |Q_\lambda| |\nabla w|    \frac{\delta \omega^2}{1 + \delta \omega},
 \end{equation}
Since  $|Q_\lambda| |\nabla w|   \frac{\delta \omega^2}{1 + \delta \omega} \le |Q_\lambda| |\nabla w| \omega$ with the latter integrable via the H\"older inequality, the Lebesgue dominated convergence used for the last summand of \eqref{eq:s43:1:4'} implies altogether
\begin{equation}\label{eq:s43:1:4''}
\limsup_{\delta \to \infty} \limsup_{k \to \infty} \int_{\Omega_j} V_\delta^{k,\lambda} \psi  \le C\epsilon+ 0 \end{equation}

Finally, let us focus on $VI^{k,\lambda}$ of \eqref{eq:s43:1:p}. We deal with it it using again biting lemma, together with the weak-$L^1$ estimate for the maximal function 
\[
 \abs{\set{M( \nabla {w^k})>\lambda}}\leq{ \frac{c\norm{\nabla {w^k}}_{L^1(\Omega)}}{\lambda}}\leq \frac{C}{\lambda},
 \]
which indicates, that here the role of the biting set $K$ may play $\set{M( \nabla {w^k})>\lambda}$ for sufficiently large $\lambda$. Indeed, in tandem with the above application of the biting Lemma, for every $j$ and $\epsilon$ there exists $\lambda_j^\epsilon$, such that
 \begin{align}
 \label{eq:wL12}
\int_{\set{M( \nabla {w^k})>\lambda}\cap \Omega_j}\abs{\nabla {w^k}}^{q'}\omega \leq \epsilon\quad \text{ for every }\quad \lambda\geq \lambda_j^\epsilon.
 \end{align}
%Thanks to \eqref{itm:weight} and  \eqref{eq:wL12}, for any  $j$ and $\epsilon$ there exists $\lambda_j^\epsilon$ that
 %\[
 %\int_{\Omega_j} \abs{ (\nabla {v^k} -\nabla v^{k,\lambda})}^2 \omega_0 \leq C     \int_{\set{M( \nabla {v^k})>\lambda}\cap \Omega_j}\abs{\nabla {v^k}}^2\omega_0  \leq \epsilon\quad \text{ for every }\quad \lambda\geq \lambda_j^\epsilon.
 %\]
Let us use Theorem~\ref{thm:liptrunc} to write
 \begin{align} \label{eq:wL12:2}
 \begin{aligned}
& \Bigabs{  \int_{\Omega_j} (Q(\nabla w^{k})   - Q(\nabla w^{k, \lambda}) )    \cdot \nabla w^{k} \;  \omega \; \psi } =   \Bigabs{   \int_{\set{M(\nabla ({w^k}))>\lambda}\cap\Omega_j}(Q(\nabla w^{k})   - Q(\nabla w^{k, \lambda}) )    \cdot \nabla w^{k} \;  \omega \; \psi  }   \\
& \le C \norm{\psi}_\infty \bigg( \int_{\Omega}\abs{\nabla w^{k,\lambda}}^{q'} \omega  + \abs{\nabla w^{k}}^{q'}\omega \bigg)^\frac{1}{q}  \bigg( \int_{\set{M(\nabla ({w^k}))>\lambda}\cap\Omega_j} \abs{\nabla ({w^k})}^{q'} \omega \bigg)^\frac{1}{q'}  , % \leq c(\epsilon K)^\frac12
\end{aligned}
 \end{align}
where, for the inequality, we used growth of $Q$.
 
 Putting together \eqref{eq:wL12:2} and \eqref{eq:wL12} we see that for every $j$ and $\epsilon$ there exists $\lambda_j^\epsilon$ such that 
% for every $\lambda>\lambda_k^\epsilon$.
  \begin{equation} \label{eq:wL12:3}
 \Bigabs{  \int_{\Omega_j}VI^{k,\lambda} \psi  }  \leq C \norm{ \phi}_\infty \epsilon^\frac{1}{q'} \qquad \text{ for every }\quad \lambda\geq \lambda_j^\epsilon.
 \end{equation}
 
Altogether, integrating  \eqref{eq:s43:1:p} over $\Omega_j$, taking in its right-hand-side \[\limsup_{\lambda \to \infty} \; \limsup_{\delta \to \infty} \; \limsup_{k \to \infty}\] and using\eqref{eq:s43:1:3}, \eqref{eq:s43:1:4''} and \eqref{eq:wL12:3}, we see that for any $j$ it holds
\begin{equation}\label{Minty:p}
Q ( \nabla w^k)   \cdot \nabla {w^k} \omega \weakto Q_0 \cdot \nabla {w} \; \omega  \qquad \textrm{weakly in } L^1(\Omega_j).
\end{equation}
%We are finished with \eqref{Minty}.

 \subsubsection{Step 4. Justifying the compactness condition \eqref{bit:p:3} via monotonicity } %Let us take any $B\in L^{q}_{\omega}(\Omega)$. Using \eqref{Minty:p},~\eqref{conn-b:p} and~\eqref{con2:p},  we get
Finally, \eqref{Minty:p} together with radial unboundedness (coercivity) and strict monotonicity of $Q$ imply
\[
\nabla {w^k} \to \nabla w \qquad \text{ a.e. in } \Omega_j
\]
For more details on this step, compare for instance pp.52-53 in the book of Roub\'i\v{c}ek \cite{RoubicekBook}.
The diagonal argument gives us a subsequence such that 
\[
\nabla {w^k} \to \nabla w \qquad \text{ a.e. in } \Omega.
\]
This, together with the \eqref{bit:p:1} implies uniform integrability, hence via the Vitali's theorem $L^1$ strong sequential precompactness of $\nabla {w^k}$. 

The proof is complete.

%\seb{Q: Do you thin this is enough? I could also copy past it from our solenoid file, but it would cost 2.5 pages.}
%\comentario{A: I think it is slightly not enough, but not in the sense of brevity of proof but the exact statement,  see my comment at the beginning of this subsection. Maybe it makes sense to move this entire Lipschitz truncation subsection to the section 3 (preliminaries) ? \seb{I think it is good put here. I like the use of it and always considered theorems for sequences rather ugly. With respect to my taste it can stay as it is. BUT PLEASE CHECK IT.}}
%\comentario{How to use it}
%Since for every $\omega\in A_2$, there exists an $s>1$, such that $L^2_\omega(\Omega)\subset L^s(\Omega)$, we can construct the following localization.
%
%For every point $x\in \Omega$, we find a ball $B\ni x$, such that $4B\subset\Omega$. We take the usual cut off function $\eta\in C^\infty_0(2B)$, with $\chi_{2B}\leq \eta\leq\chi_B$ and find that $\tilde{u}=\eta u-\Bog(\nabla \eta\cdot u)\in W^{1,s}_{0,\divergence}(2B)$, therefore we can apply Theorem~\ref{thm:liptrunc} on $\tilde{u}$, and $\tilde{u}^\lambda$, will be an admissable testfunction.
%
%The continuity of the Bogovski{\u\i} operator in $L^s$ can be found in 
%\cite{Bog80,DieRS10}  and in the Schumacher Paper or PhD thesis~\cite{Sch07} for weighted spaces.
% \comentario{J: end of copy-paste from 'homework'}

\section{Proofs of main results}\label{sec:ex}
This section is dedicated to the proofs of  Theorem~\ref{th},  Theorem~\ref{th1} and Corollary \ref{cor1}. Theorem~\ref{th} is a special case of Theorem~\ref{th1}, so let us focus on the latter. The main ingredients of its proof are a priori estimates provided by Theorem~\ref{thm:a-priori}, limit identification by Theorem \ref{T5} and weighed considerations that allow to provide optimal regularity.

\subsection{Existence. Step 1. Approximate problems}\label{ssec:ap}
Recall that an arbitrary  $f\in L^q_\omega (\Omega)$ with  $\omega\in A_q$, $1<q<\infty$, is a force of the considered problem \eqref{eq:sysA}.
We have by \eqref{eq:lqprop}, that $f\in L^{s_0} (\Omega)$ for an $s_0 \in (1,2)$. Formula~\eqref{weight:s} with $\alpha = 2-s_0$ implies that $ (Mf)^{s_0-2} \in A_2$, hence also $\omega_0 := (1+ Mf)^{s_0-2}$ belongs to $A_2$.  Conseqently we have $f \in L^2_{\omega_0}(\Omega)$, compare \eqref{weight1}.

Let us  define $f^k:=f \chi_{\set{|f|<k}}$.
Then 
\begin{align}\label{cfn}
  f^k \to f \quad \textrm{ strongly in } L^2_{\omega_0} (\Omega) \cap
  L^{{s_0}} (\Omega) \cap L^q_{\omega}  (\Omega)
\end{align}
For our $f^k \in L^2  (\Omega)$ we can use the standard monotone operator theory to find $v^k \in W^{1,2}_{0} (\Omega)$ satisfying
\begin{equation}
  \int_{\Omega} \Scal(x,\varepsilon (v^k)) \cdot \nabla \phi  = \int_{\Omega}f^k
  \cdot \nabla \phi    \qquad \textrm{ for all } \phi\in W^{1,2}_{0, \divv}(\Omega). \label{wfn}
\end{equation}
It is equivalent to finding $(v^k, \pi^k) \in W^{1,2}_{0} (\Omega) \times \hL^{2}(\Omega) $ solving weakly \eqref{eq:sysA}. %, in view of \comentario{J: cite here some deRham bullshit \seb{See above.}}
%\subsection{Step 2. Uniform estimate}
%Let us rewrite \eqref{wfn} as a weak formulation of the linear Stokes problem
%\begin{equation}
%\int_{{\Omega}} \varepsilon ({v^k}) \cdot \nabla \phi  = \int_{{\Omega}}(f^k -\Scal(x,\varepsilon ({v^k}))  +\varepsilon ({v^k}))\cdot \nabla \phi. \label{wfn2}
%\end{equation}
%Since $\nabla {v^k} \in L^2 \subset L^2_{{\omega_0}}$ we can use estimate of Lemma~\ref{lem:CZ} and Assumption \ref{ass:A} to write
%\begin{align*}
%& | \nabla {v^k}|^2_{L^2_{\omega_0} (\Omega)} \leq C \int_{\Omega}\abs{f^k}^2 {\omega_0} + \abs{\Scal (x,\varepsilon ({v^k}))- \varepsilon ({v^k})}^2{\omega_0}\\
%&\leq C\int_{\Omega}(\abs{f}^2+4c^2_1 m^2) \, {\omega_0} + C \int_{\{|\varepsilon ({v^k})|\ge m \}}\frac{\abs{\Scal (x,\varepsilon ({v^k}))- \varepsilon ({v^k})}^2}{|\varepsilon ({v^k})|^2} |\varepsilon ({v^k})|^2{\omega_0}.
%\end{align*}
%Due to the assumed linearity-at-infinity we can find such $m = m_0$ that the last summand on the r.h.s. above does not exceed half of the l.h.s. Consequently
%\begin{align}\label{pre_finaln}
% |\nabla {v^k}|_{L^2_{\omega_0} (\Omega)} \leq C \left(1+  |f|_{L^2_{\omega_0} (\Omega)} \right).
%\end{align}
%with $k$-uniform $C$. Due to our choice of $\omega_0$ and estimate for maximal functions, we see that $|f|_{L^2_{\omega_0} (\Omega)} \le C |f|_{L^q (\Omega)}$. This and \eqref{eq:lqprop} allows to write
By Theorem~\ref{thm:a-priori} (used three times, for  $L^q_\omega (\Omega), L^{s_0}(\Omega)$ and for $L^2_{\omega_0} (\Omega)$), we find that uniformly in $k$
\begin{equation}\label{finaln}
\begin{aligned}
 \norm{\nabla v^k}_{L^q_{\omega} (\Omega)} +  \norm{\pi^k}_{L^q_{\omega} (\Omega)}  \leq C (1+ \norm{f^k}_{L^q_\omega (\Omega)} ) \leq C (1+ \norm{f}_{L^q_\omega (\Omega)} ), \\
 \norm{\nabla {v^k}}_{L^2_{\omega_0} (\Omega)} +  \norm{\pi^k}_{L^2_{\omega_0} (\Omega)} +   \norm{\nabla {v^k}}_{L^{s_0} (\Omega)}  + \norm{\pi^k}_{L^{s_0} (\Omega)} \leq C (1 + \norm{f^k}_{L^2_{\omega_0} (\Omega)}) \leq C_f.
\end{aligned}
\end{equation}
%where a  small $s > 1$ of  is implied by  \eqref{eq:lqprop}.

\subsection{Existence. Step 2. Limit passage}
Using the estimate~\eqref{finaln}, the reflexivity of the corresponding spaces, the unique identification of the limit~$v$ in $W^{1,1}(\Omega)$ and the growth of Assumption \ref{ass:A},  we obtain for a (non-relabeled) subsequence
\begin{align}
  \label{conn-a}
  {v^k} &\rightharpoonup v &&\textrm{weakly in } W^{1,s_0}_0(\Omega),
  \\
  \label{conn-b}
(\nabla {v^k}, \pi^k)  &\rightharpoonup (\nabla v, \pi) &&\textrm{weakly in }
  L^2_{\omega_0} (\Omega) \cap L^{s_0}(\Omega) \cap   L^q_{\omega} (\Omega)  ,
  \\
  \Scal (x,\varepsilon ({v^k})) &\rightharpoonup \Scal_0 &&\textrm{weakly in }
  L^2_{\omega_0} (\Omega) \cap L^{s_0}(\Omega) \cap   L^q_{\omega} (\Omega)   \label{con2}.
\end{align}

% Hence, having \eqref{finaln}--\eqref{estkeyn2}, we can extract a subsequence that we do not relabel such that
% \begin{align}
%   {v^k} &\rightharpoonup u &&\textrm{weakly in }
%   W^{1,\frac{2}{p}}_0(\Omega; \mathbb{R}^k)\label{con1},
%   \\
%   A(x,\varepsilon ({v^k})) &\rightharpoonup \overline{A} &&\textrm{weakly in }
%   L^{\frac{2}{p}}(\Omega; \mathbb{R}^{n\times N})\label{con2}
%   \\
%   \varepsilon ({v^k}) \sqrt{w} &\rightharpoonup \nabla u \sqrt{w}
%   &&\textrm{weakly in } L^{2}(\Omega; \mathbb{R}^{n\times
%     N})\label{con3},
%   \\
%   A(x,\varepsilon ({v^k}))\sqrt{\omega} &\rightharpoonup
%   \overline{A}\sqrt{\omega} &&\textrm{weakly in } L^{2}(\Omega;
%   \mathbb{R}^{n\times N}).\label{con4}
% \end{align}
Hence the  lower weak semicontinuity implies via \eqref{finaln}
\begin{equation}\label{eq:final}
\begin{aligned}
 \norm{\nabla v}_{L^q_{\omega} (\Omega)} +  \norm{\pi}_{L^q_{\omega} (\Omega)}  &\leq C (1+ \norm{f}_{L^q_\omega (\Omega)} ) \\
  \norm{\nabla v}_{L^{s_0} (\Omega)}  + \norm{\nabla v}_{L^2_{\omega_0} (\Omega)} &\leq C_f.
\end{aligned}
\end{equation}
Convergences \eqref{con2} and \eqref{cfn} used in \eqref{wfn} imply
\begin{align}
  \label{eq:limit}
  \int_\Omega \Scal_0 \cdot\nabla \phi=\int_\Omega f\cdot\nabla \phi \qquad\text{ for all } \phi \in W^{1,{\infty}}_{0, \divv}(\Omega).
\end{align}
%\comentario{J: maybe better use $\infty$ in place of $(s \wedge s_0)'$ above? seb{Did that.}}
Hence, to complete the proof of Theorem~\ref{th1}, it remains to identify the limit properly, \emph{i.e.} to show 
\begin{align}
\Scal_0 (x) &=\Scal (x,\nabla v(x)) \qquad\textrm{in  } \Omega,\label{show22}
\end{align}
because then the optimal regularity will be given by the first line of \eqref{eq:final}.
\subsection{Existence. Step 3. Limit identification}
This is the central part of our proof. Its crucial part will follow from the solenoidal, weighted, biting div-curl lemma, i.e. Theorem \ref{T5}.  

Recall that the classical way of identifying the limit in nonlinear problems, namely use of monotonicity and dealing with the most nonlinear part via the equation, is impossible in our very weak setting, since one cannot use $u$ as a test function in \eqref{eq:limit}. 

Observe also that taking the weighed limits is  crucial to end up with optimal regularity related to $f$ (recall the our weight $\omega_0$ is related to $Mf$).

Let us use Theorem \ref{T5} with the following choices
\[
q=q'=2, \quad \omega = \omega_0, \quad  a^k = \nabla v^k, \quad s^k = \Scal (\cdot,\varepsilon ({v^k})).
\]
The uniform boundedness assumption \eqref{bit3} is satisfied thanks to \eqref{finaln}. The compactness assumption  \eqref{bit4} holds thanks to the weak formulation \eqref{wfn} with $c^k$ as the test function. Finally, the compensation assumptions \eqref{bit5}, \eqref{bit6} hold automatically, since our $ a^k $ is a gradient of a solenoidal function.

Thesis of Theorem \ref{T5} provides thence, for a non-relabelled subsequence and a non-decreasing sequence of measurable subsets
  $\Omega_j\subset\Omega$ with $|\Omega \setminus \Omega_j|\to 0$ as $j\to
  \infty$, that
\begin{equation}\label{Minty}
 \Scal (\cdot,\varepsilon ({v^k}))\cdot\nabla {v^k} \omega_0 \weakto  \Scal_0
  \cdot  \nabla v \, \omega_0 \qquad \textrm{weakly in } L^1(\Omega_j).
\end{equation}

The last needed step: from \eqref{Minty} to \eqref{show22}, will be performed via monotonicity. Let us take any $B\in L^{2}_{\omega_0}(\Omega)$. Using \eqref{Minty},~\eqref{conn-b} and~\eqref{con2},  we get
\begin{equation}\label{Minty2}
  (\Scal(x,\varepsilon ({v^k}))-\Scal(x,B)) \cdot (\nabla {v^k}-B) \,\omega_0 \rightharpoonup
  ({\Scal_0}-\Scal(x,B)) \cdot (\nabla u-B)\, \omega_0 \quad
  \textrm{weakly in }  L^1(\Omega_j).
\end{equation}
Monotonicity of $\Scal$ implies that the limit is signed as well, thus
\begin{equation}\label{Minty3}
  \int_{\Omega_j} ({\Scal_0}-\Scal(x,B)) \cdot (\nabla v-B)\, \omega_0 \dx \ge 0
\end{equation}
for any $  j\in \mathbb{N}$. Consequently
\begin{equation*}
\infty> \int_{\Omega}(\Scal_0-\Scal(x,B)) \cdot (\nabla v-B)\, \omega_0  \ge \int_{\Omega\setminus \Omega_j}(\Scal_0-\Scal(x,B)) \cdot (\nabla v-B)\, \omega_0 
\end{equation*}
Observe that the integrals above are well defined due
to~\eqref{conn-b}, \eqref{con2} and the assumed growth of $\Scal$. Therefore, recalling that $|\Omega
\setminus \Omega_j|\to 0$ as $j\to \infty$, we let $j \to \infty$ and obtain
\begin{equation*}
\infty> \int_{\Omega}(\Scal_0-\Scal(x,B)) \cdot (\nabla v-B)\, \omega_0 \dx \ge 0 \qquad \textrm{for all } B\in L^2_{\omega_0}(\Omega).
\end{equation*}
Choosing $B:=\nabla u -\varepsilon G$ with an  arbitrary $G\in L^{\infty}(\Omega)$, we get
\begin{equation*}%\label{Minty4}
\infty>  \int_{\Omega}(\Scal_0-\Scal(x,\nabla v - \varepsilon G)) \cdot G\, \omega_0 \dx \ge 0 
\end{equation*}
Finally, using the Lebesgue dominated convergence theorem, Assumption \ref{ass:A} (growth and continuity), we let $\varepsilon \to 0_+$ to deduce
\begin{equation*}
\int_{\Omega}(\Scal_0-\Scal(x,\nabla v)) \cdot G\, \omega_0 \dx \ge 0.
\end{equation*}
Choosing 
$$
G:=-\frac{\Scal_0-\Scal(x,\nabla v)}{1+\abs{\Scal_0-\Scal(x,\nabla v)}}.
$$
and utilizing that $\omega_0$ is strictly positive almost everywhere in $\Omega$, we arrive at validity of \eqref{show22} a.e. in $\Omega$. Consequently
\begin{align}
  \label{eq:existence}
  \int_\Omega \Scal(x,\nabla v) \cdot\nabla \phi=\int_\Omega f\cdot\nabla \phi \qquad\text{ for all } \phi \in W^{1,\infty}_{0, \divv}(\Omega).
\end{align}
with estimate  \eqref{eq:final}.

We have ended the proof of the existence part of Theorem \ref{th1}. The estimate \eqref{eq:th1} is given by Theorem \ref{thm:a-priori}. Hence, to conclude proof of Theorem \ref{th1},  we are left with showing its uniqueness statement.

\subsection{Uniqueness}
Recall that now the tensor $\Scal$ satisfies additionally Assumption \ref{ass:B}.
A difference between two solutions $u_1$ and $u_2$ to \eqref{eq:sysA} with the same force $f \in L^q_\omega (\Omega)$ satisfies
\begin{equation}
\label{un1}
\int_{\Omega} \big(\Scal(x,\varepsilon (v_1) )-\Scal(x,\varepsilon  (v_2) ) \big)\cdot \nabla \varphi \dx = 0, % \qquad \textrm{ for all } \varphi\in C^{0,1}_{0, \divv} (\Omega).
\end{equation}
with the admissible class of $\varphi$ dictated by the optimal  $L^q_\omega$-regularity of $v_1, v_2$, see \eqref{eq:th1}.
Hence, if we could have chosen $\varphi = v_1 - v_2$, the assumed strict monotonicity would imply  $v_1 = v_2$. Therefore in the case $L^{q}_\omega(\Omega)\subset L^2(\Omega)$ the proof is finished. But generally, we find that $f\in L^{s_0}(\Omega)$, merely for some $s_0 \in (1,2]$, compare Section \ref{ssec:ap}. Such $L^{s_0}$-regularity seems insufficient, since possibly $s_0<2$. {Nevertheless, we will be able to show that $\nabla(v_1 - v_2) \in L^2 (\Omega)$  via the weighted estimates and conclude the uniqueness using this extra regularity for the difference.

 To begin with, let us recall} that $f\in L^2_{\omega_0}(\Omega)$ for $\omega_0 = (1+Mf)^{s_0-2}$ and therefore also $\nabla v_1, \nabla v_2\in L^2_{\omega_0}(\Omega)$.
Let us rewrite the identity \eqref{un1} into the form
\begin{equation}
\label{un2}
\int_{\Omega} (\varepsilon (v_1) - \varepsilon (v_2))  \nabla \varphi 
= \mu^{-1} \int_{\Omega} \big(\mu \;  \varepsilon (v_1) - \Scal \big(x, \varepsilon (v_1) \big) - \big(\mu \; \varepsilon v_2-\Scal (x, \varepsilon (v_2))\big)  \nabla \varphi ,
\end{equation}
which is valid for all $\varphi\in W^{1,\infty}_{0, \divv}(\Omega)$.

 Let $w^j:=\min\set{1, (j\omega_0)}$ and observe that $\nabla v \in L^2_{\omega_j}(\Omega)$ for a fixed $j$, since $\nabla v \in L^2_{\omega_0}(\Omega)$ in view of the previous subsection. Moreover, $A_p (\omega_j) \le \max (1,A_p (\omega_0))$ in view of definition \ref{defAp2}. Consequently, we can use the linear maximal regularity Lemma~\ref{lem:CZ} to obtain
\begin{equation}
\label{un3}
\int_{\Omega}|\varepsilon (v_1) - \varepsilon (v_2) |^2 \omega^{{j}} \le C \mu^{-1} \int_{\Omega} \left| \mu \; \varepsilon u_1  - \Scal(x,\varepsilon (v_1) ) -\big( \mu \; \varepsilon (v_2) -\Scal(x, \varepsilon (v_2) ) \big)\right|^2 \omega^j 
\end{equation}
with finite r.h.s. and ${j}$-independent  $C$ of, the latter due to $A_p (\omega_j) \le \max (1,A_p (\omega_0))$.
Next, using the estimate \eqref{algebra} of Lemma \ref{L:algebra}  in \eqref{un3}, we find that for any $\delta>0$
\begin{equation}
\label{un4}
\begin{split}
\int_{\Omega}|\varepsilon (v_1) -\varepsilon (v_2) |^2 \omega^j \le C  \mu^{-1} \delta \int_{\Omega}| \varepsilon (v_1) - \varepsilon (v_2) |^2\omega^j + C(\delta)\omega^j.
\end{split}
\end{equation}
Thus, setting $\delta:=\frac{\mu}{2C}$ yields
\begin{equation}
\label{un5}
\begin{split}
\int_{\Omega}| \varepsilon (v_1 - v_2)|^2 \omega^j \le C(\delta) \int_{\Omega} \omega^j \le C,
\end{split}
\end{equation}
where the last inequality follows from the fact that $\Omega$ is bounded and $\omega^j \le 1$. Hence, letting $j\to \infty$ in \eqref{un5}, together with $\omega^j \nearrow 1$ (which follows from the fact that $\omega_0>0$  almost everywhere) and the monotone convergence theorem implies
$$
\int_{\Omega}| \varepsilon (v_1 - v_2)|^2   \le C.
$$
Hence, via the Korn inequality, we see that  $v_1-v_2 \in W^{1,2}_0(\Omega)$. Consequently, using  growth of $\Scal$ given by Assumption \ref{ass:A},  we have that
\[
\int_{\Omega}|\Scal (x, \varepsilon (v_1))-\Scal(x, \varepsilon (v_2))|^2 \le C.
\]
Therefore, \eqref{un1} holds for all $\varphi\in W^{1,2}_{0, \divv} (\Omega)$, including $\varphi:=v_1-v_2$. The strict monotonicity finishes the proof of the uniqueness. The entire 
Theorem~\ref{th1} is proved.
\qed

\subsection{Proof of {Corollary~\ref{cor1}}} 
It follows the lines of proof of Theorem~\ref{th1}, with rather straightforward modifications related to involved inhomogeneities. More precisely, in Steps 1 and 2 of proof of Theorem~\ref{th1} we use now the inhomogenous estimate of Theorem~\ref{thm:a-priori}. It implies weak convergence in the respective spaces. To identify the limit (reconstruct the stress tensor) along Step 3, its arguments can be shown for $v-\gamma^{-1}(g)-\Bog(v-\gamma^{-1}(g))$, because the  appearing extra terms are converging due to the weak-strong coupling. Substeps 4.3, 4.4 can be then adapted immediately. The proof of the uniqueness is line by line the same.
%For the linear case, we use the information provided by  Theorem~\ref{th} qualitatively, which implies via Lemma~\ref{lem:CZ} the right quantitative estimates, which  concludes the proof of Theorem~\ref{th2}. \comentario{J: I suggest throwing out  Theorem~\ref{th2}, since  Lemma~\ref{lem:CZ} has only identity matrix in the main part, so consequently we are getting, via procedure explained in the previous sentence, again non homogenous estimate. On the other hand, since you suggested (what I did, why I am listening to you:)?) that we write Lemma Lemma~\ref{lem:CZ}  for inhomogenous data, we can write a version of Theorem~\ref{th2} for the sake of this inhomogenous data. Then either we write identity matrix and get homogenous estimate or write $A(x)$ that complies with our assumptions (i.e. is identity at infinity) and get inhomogenous estimates \seb{ I suggest to remove it and make a remark on continous coeficients.}}

\bibliographystyle{abbrv} %{amsalpha}

\bibliography{burczak}
\end{document}